\documentclass{amsart}

\usepackage{mathrsfs}
\usepackage{amsfonts}
\usepackage{amssymb,amsmath}
\usepackage{amsthm}

\usepackage{graphics}
\usepackage{graphicx}
\usepackage{epsfig}
\usepackage[bookmarks=true]{hyperref}
\usepackage{multirow} 
\usepackage{xcolor}
\usepackage{makecell}

\newtheorem{theorem}{Theorem}[section]
\newtheorem{lemma}[theorem]{Lemma}

\theoremstyle{definition}
\newtheorem{definition}[theorem]{Definition}

\theoremstyle{remark}
\newtheorem{remark}[theorem]{Remark}

\numberwithin{equation}{section}

\hypersetup{hidelinks}

\usepackage{changes}

\begin{document}

\title[Posterior contraction for empirical Bayesian approach]
 {Posterior contraction for empirical Bayesian approach to inverse problems under non-diagonal assumption}

\author[J. Jia]{Junxiong Jia}
\address{School of Mathematics and Statistics,
Xi'an Jiaotong University,
Xi'an,
710049, China}
\email{jjx323@xjtu.edu.cn}


\author[J. Peng]{Jigen Peng*}
\address{School of Mathematics and Information Science,
Guangzhou University,
Guangzhou,
510006, China }
\email{jgpeng@mail.xjtu.edu.cn}

\author[J. Gao]{Jinghuai Gao}
\address{School of Electronic and Information Engineering,
Xi'an Jiaotong University,
 Xi'an,
710049, China}
\email{jhgao@mail.xjtu.edu.cn}






\begin{abstract}
We investigate an empirical Bayesian nonparametric approach to a family of linear inverse problems with Gaussian prior and Gaussian noise.
We consider a class of Gaussian prior probability measures with covariance operator indexed by a hyperparameter that quantifies regularity.
By introducing two auxiliary problems, we construct an empirical Bayes method and prove that this method can automatically select the hyperparameter.
In addition, we show that this adaptive Bayes procedure provides optimal contraction rates up to a slowly varying term and an arbitrarily small constant,
without knowledge about the regularity index. Our method needs not the prior covariance, noise covariance and forward operator have a
common basis in their singular value decomposition, enlarging the application range compared with the existing results.
A simple simulation example is given that illustrates the effectiveness of the proposed method. 
\end{abstract}

\maketitle



\section{Introduction}

Inverse problems for partial differential equations arise naturally from medical imaging, seismic exploration and so on.
There are two difficulties for such types of problems: one is non-uniqueness and another one is instability.
Bayesian approach formulates inverse problems as statistical inference problems, which enables us to overcome both of these difficulties.
Because of that, Bayes' inverse method has attracted a lot of interest in recent years; see for instance
\cite{Burger2014IP,inverse_fluid_equation,MAPSmall2013,Dunlop2016IP,Junxiong2016IP,Junxiong2016,Koponen2014IEEE,Stuart2010AN}.
However, compared with the classical regularization techniques, the development of a theory of Bayesian posterior consistency
is still in its infancy.

For clarity, let us provide some basic settings.
Let $H$ be a separable Hilbert space, with norm $\|\cdot\|$ and inner product $(\cdot, \cdot)$, and let
$\mathcal{A}: \mathcal{D}(\mathcal{A})\subset H \rightarrow H$ be a self-adjoint and positive-definite linear operator with bounded inverse.
Then, we usually assume the following problem
\begin{align}\label{S11}
y = \mathcal{A}^{-1}u + \frac{1}{\sqrt{n}}\xi,
\end{align}
where $\frac{1}{\sqrt{n}}\xi$ is noise and $y$ is a noisy observation of $\mathcal{A}^{-1}u$.
The inverse problem can be thought to find $u$ from $y$.
Following the framework shown in \cite{Dashti2017}, we denote $\mu^{0}$ as the prior measure and $\mathbb{P}_{\text{noise}}$ as the noise measure and assume that
\begin{itemize}
  \item Prior:  $u \sim \mu^{0}$,
  \item Noise:  $\xi \sim \mathbb{P}_{\text{noise}} = \mathcal{N}(0,\Sigma)$,
\end{itemize}
where $\mathcal{N}(0,\Sigma)$ is a Gaussian measure with zero mean and covariance operator $\Sigma$.
Let $\mu^{y}_{n}$ be the posterior measure which is the solution of the inverse problem.
Usually, algorithms, such as Markov chain Monte Carlo (MCMC), can be employed to probe the posterior probability measure \cite{Cotter2013SS}.
If we assume that the range of $\mathcal{A}^{-1}$ is included in the domain of $\mathcal{C}^{-1/2}$, then (\ref{S11}) can be rewritten as follows:
\begin{align}\label{S14}
d = \mathcal{T}u + \frac{1}{\sqrt{n}}\eta,
\end{align}
where $d = \Sigma^{-1/2}y$, $\mathcal{T} = \Sigma^{-1/2}\mathcal{A}^{-1}$ and $\eta = \Sigma^{-1/2}\xi$.
Obviously, we have $\eta \sim \mathcal{N}(0,I)$ that is a white Gaussian noise.
Let $\mu_{n}^{d}$ be the posterior measure and still let $\mu^{0}$ denote the prior measure.
Here we will be particularly interested in the small noise limit where $n \rightarrow \infty$.
Under the frequentist setting, the data $d = d^{\dag}$ are generated by
\begin{align}\label{S12}
d^{\dag} = \mathcal{T}u^{\dag} + \frac{1}{\sqrt{n}}\eta, \quad \eta \sim \mathcal{N}(0,I),
\end{align}
where $u^{\dag}$ is a fixed element of $H$.
Our aim is to show that the posterior probability measure $\mu_{n}^{d}$ contracts to a
Dirac measure centered on the fixed true solution $u^{\dag}$.
In the following, let $\mathcal{L}(\eta)$ denote the law of a random variable $\eta$ and $\mathbb{E}_{0}$ denote
the expectation corresponding to the distribution of the data $d^{\dag}$.
Let $\epsilon_n$ be a positive sequence that converges to $0$, 
then a formal description of the concept of posterior contraction with rate $\epsilon_n$ can be stated as follows.

\textbf{Definition 1}. A sequence of Bayesian inverse problems $(\mu^{0}, \mathcal{T}, \mathcal{L}(\frac{1}{\sqrt{n}}\eta))$
is posterior consistent for $u^{\dag}$ with rate $\epsilon_{n} \downarrow 0$ if for (\ref{S12}), and
for every positive sequence $M_{n}$ such that
\begin{align}\label{S13}
\mathbb{E}_{0}\left(\mu_{n}^{d}\left\{u:\, \|u-u^{\dag}\| \geq M_{n}\epsilon_{n}\right\}\right) \rightarrow 0,
\quad \forall \,\, M_{n} \rightarrow \infty.
\end{align}

Until now, there are a number of studies for posterior consistency and inconsistency for Bayes' inverse method.
When the forward operator $\mathcal{A}^{-1}$ is a nonlinear operator, Vollmer \cite{Vollmer2013IP} provides a general framework and studies
an inverse problem for elliptic equation in detail.
Recently, Knapik and Salomond \cite{Salomond2018B} provide a general framework for investigating the posterior consistency of linear inverse problems,
which allows non-Gaussian priors.
However, due to the difficulties brought by nonlinearity and non-Gaussian,
the existing results are mainly focus on the situation that the forward operator $\mathcal{A}^{-1}$ is linear,
the prior measure is Gaussian $\mu^{0} = \mathcal{N}(0,\mathcal{C})$ and the noise measure is also Gaussian $\mathbb{P}_{\text{noise}} = \mathcal{N}(0,\Sigma)$.
When $\mathcal{A}^{-1}$, $\mathcal{C}$ and $\Sigma$ are all simultaneous diagonalizable,
Knapik et al \cite{Knapik2011AS} provide a roadmap for what is to be expected regarding posterior consistency.
This work reveals an important fact that is the optimal convergence rates can be obtained if and only if the regularity of the prior and the truth are matched.
Therefore, Knapik et al \cite{Knapik2016} propose an empirical Bayes procedure which provides an estimate of the regularity of the prior
through the data. By choosing the regularity index adaptively, the optimal convergence rates are obtained up to a slowly varying factor.
Later, Szab\'{o} et al \cite{Szabo2015AS} introduce the ``polished tail'' condition and investigate the frequentist coverage of Bayesian credible sets
by choosing the regularity of the prior through an empirical Bayes method.
Recently, by employing abstract tools from regularization theory,
Agapiou and Math\'{e} \cite{Agapiou2018S} study the posterior consistency by choosing a non-centered prior through empirical Bayes method.
In 2018, Trabs \cite{Trabs2018IP} obtain optimal convergence rates up to logarithmic factors when the forward linear operator depends on an unknown parameter.

When $\mathcal{A}^{-1}$, $\mathcal{C}$ and $\Sigma$ are not simultaneous diagonalizable,
optimal contraction rates cannot be obtained by employing similar methods developed for the simultaneous diagonalizable case.
Concerning this case, theories about partial differential equation (PDE) have been employed to obtain
nearly optimal convergence rates in \cite{Agapiou2013IP} by using precision operators.
Using properties of the hypoelliptic operators, Kekkonen et al \cite{Kekkonen2016IP} obtain nearly optimal convergence rates for Bayesian inversion with hypoelliptic operators.
Recently, Math\'{e} \cite{Mathe2018arXiv} studies the posterior consistency for non-commuting operators
under low-smoothness assumptions by employing
the ideas of link conditions originating from classical regularization method in Hilbert scales.
Besides these studies, to the best of our knowledge, there are little investigations of the posterior consistency for the non-simultaneous diagonalizable case.
In order to achieve the optimal contraction rates, the regularity of the true solution should be known in the above mentioned investigations.
Hence, as discussed in the last part of \cite{Agapiou2013IP},
how to generalize the empirical Bayes procedure developed for the simultaneous diagonalizable case to the non-diagonal case is crucial
for the applicability of the posterior consistent theory.

For the simultaneous diagonalizable case, the structures of the posterior probability measure for
the regularity index can be analyzed in detail.
For the non-diagonal case, the determination of the contraction rates cannot be reduced to bounding sums,
and, in particular, the posterior probability measure for the regularity index is hard to define.
The reason is that probability density functions cannot be defined in infinite-dimensional space and some integrals cannot be calculated easily
to obtain the log-likelihood of the regularity index obtained in \cite{Knapik2016}.
Hence, we can hardly define an empirical Bayes procedure for the non-diagonal case.
To overcome these difficulties, we notice that regularity index is derived from data and the posterior contraction rates are
only depending on the regularity index of the prior measure and the true function $u^{\dag}$.
So, if we introduce some artificial diagonal problem with the same regularity properties as for the non-diagonal problem, we may
obtain optimal posterior contraction rates relying on the relations between the artificial diagonal problem and the non-diagonal problem.

For a positive constant $\alpha$, we consider Gaussian prior measure $\mu^{0} = \mathcal{N}(0,\mathcal{C}^{\alpha})$.
Firstly, we consider the following problem
\begin{align}\label{S15}
d = m + \frac{1}{\sqrt{n}}\eta,
\end{align}
where $m := \mathcal{T}u$. Then the prior measure for $m$ obviously is $\mathcal{N}(0,\mathcal{T}\mathcal{C}^{\alpha}\mathcal{T}^{*})$.
Concerning this problem, the forward operator and the covariance operator of the noise are all the identity operator $I$.
Hence, the posterior contraction rates of the problem (\ref{S15}) seems can be obtained by the ideas developed for the diagonal case.
However, the operator $\mathcal{T}$ appears in the prior probability measure, which makes difficult to derive estimations of the corresponding eigenvalues.
By constructing an artificial prior probability measure with similar regularity properties as for the prior probability measure,
we can construct maximum likelihood-based estimate for the regularity index and prove the optimal contraction rates for (\ref{S15})
up to a slowly varying term. 
More specifically, we introduce ``self-similar'' condition to ensure an upper bound of the regularity index calculated by data,
which seems necessary for estimating various quantities, i.e., 
differences of conditional mean estimates for problem (\ref{S15}) and the artificial diagonal problem.
With these estimates, we can obtain nearly optimal contraction rate for problem (\ref{S15}) based on the 
results of artificial diagonal problem. 
Details are shown in Lemma \ref{lemmaUpDownEst1}, Subsection \ref{tranSection} and the proof of Theorem \ref{converTheo2}.
At last, we transform the results of problem (\ref{S15}) to original problem (\ref{S14}) which can be achieved by employing
the method developed in \cite{Salomond2018B}.

The outline of this paper is as follows. In Section \ref{BasicSetting}, we provide a brief introduction about Hilbert scales and
give some essential assumptions about the covariance operators of the prior and noise probability measure.
In Section \ref{PosteriorCon}, for illustrating the ideas clearly, we present our main results without proofs.
In Section \ref{ExampleSection}, two examples are given, which illustrate the usefulness of our results.
In Section \ref{NumSection}, we provide some numerical results of Example 1 given in Section \ref{ExampleSection}, 
which illustrate the effectiveness of the proposed method. 
In Section \ref{ConSection}, we provide a brief summary and outlook.
In Section \ref{proofSection}, we collect all of the proof details.


\section{Basic settings and assumptions}\label{BasicSetting}

In this section we present basic settings and show some important assumptions made in this paper.
For the reader's convenience, let us provide an explanation for some frequently used notations firstly.

\textbf{Notations:}
\begin{itemize}
\item The set of all bounded linear operators mapping from some Hilbert space $H$ to $H$ is denoted by $\mathcal{B}(H)$,
and the corresponding operator norm is denoted by $\|\cdot\|_{\mathcal{B}(H)}$.
\item The range of some operator $\mathcal{C}$ is denoted by $\mathcal{R}(\mathcal{C})$, 
and the domain of the operator $\mathcal{C}$ is denoted by $\mathcal{D}(\mathcal{C})$.
\item For a linear operator $\mathcal{T}$, its dual operator is denoted by $\mathcal{T}^{*}$.
\item The notation $\mathbb{E}_{0}$ stands for the expectation corresponding to the distribution of the data $d^{\dag}$ generated from the truth.
\item For two sequences $a_{n}$ and $b_{n}$ of numbers, we denote by $a_{n} \lesssim b_{n}$ ($a_{n} \gtrsim b_{n}$) that there is $M\in \mathbb{R}$
such that $a_{n} \leq M b_{n}$ ($a_{n} \geq M b_{n}$) for $n$ large enough.
If $a_{n} \lesssim b_{n} \lesssim a_{n}$, we write $a_{n} \asymp b_{n}$. 
For two Hilbert spaces $H_1$ and $H_2$, the notation $\|\cdot\|_{H_1}\lesssim\|\cdot\|_{H_2}$
means that there is $M\in \mathbb{R}$ such that $\|\cdot\|_{H_1}\leq M\|\cdot\|_{H_2}$. The notations $\gtrsim$ and $\asymp$
for Hilbert space setting are also similar to the sequence case, which will be used frequently.
\end{itemize}

\subsection{Preliminaries}
Firstly, we will present a brief introduction about Hilbert scales \cite{Engl1996Book} which provides powerful tools for measuring
the smoothness of the noise, the forward operator and the samples of the prior.
Let $\mathcal{C}:\, H \rightarrow H$ be a self-adjoint, positive-definite, trace class, linear operator
with eigensystem $(\lambda_{i}^{2},\phi_{i})_{i=1}^{\infty}$.
Considering $H = \overline{\mathcal{R}(\mathcal{C})}\oplus\mathcal{R}(\mathcal{C})^{\perp} = \overline{\mathcal{R}(\mathcal{C})}$,
we know that $\mathcal{C}^{-1}$ is a densely defined, unbounded, symmetric and positive-definite operator in $H$.
Let $\|\cdot\|$ be the norm defined on the Hilbert space $H$.
We define the Hilbert scales $(\mathcal{H}^{t})_{t\in\mathbb{R}}$,
with $\mathcal{H}^{t} := \overline{\mathcal{S}_{f}}^{\|\cdot\|_{\mathcal{H}^{t}}}$, where
\begin{align*}
\mathcal{S}_{f} := \bigcap_{n=0}^{\infty}\mathcal{D}(\mathcal{C}^{-n}), \quad
( u, v )_{\mathcal{H}^{t}} := ( \mathcal{C}^{-t/2}u, \mathcal{C}^{-t/2}v ), \quad
\|u\|_{\mathcal{H}^{t}} := \left\| \mathcal{C}^{-t/2}u \right\|.
\end{align*}

In addition, the norms defined above possess the following properties.
\begin{lemma}\label{L2Inter}
(Proposition 8.19 in \cite{Engl1996Book})
Let $(\mathcal{H}^{t})_{t\in\mathbb{R}}$ be the Hilbert scale induced by the operator $\mathcal{C}$ defined above.
Then the following assertions hold:
\begin{enumerate}
  \item Let $-\infty < s < t < \infty$. Then the space $\mathcal{H}^{t}$ is densely and continuously embedded in $\mathcal{H}^{s}$.
  \item If $t \geq 0$, then $\mathcal{H}^{t} = \mathcal{D}(\mathcal{C}^{-t/2})$ and $\mathcal{H}^{-t}$ is the dual space of $\mathcal{H}^{t}$.
  \item Let $-\infty < q < r < s < \infty$ then the following interpolation inequality holds
  \begin{align*}
  \|u\|_{\mathcal{H}^{r}} \leq \|u\|_{\mathcal{H}^{q}}^{\frac{s-r}{s-q}} \|u\|_{\mathcal{H}^{s}}^{\frac{r-q}{s-q}},
  \end{align*}
  where $u \in \mathcal{H}^{s}$.
\end{enumerate}
\end{lemma}

Next, let us provide some necessary notations about infinite sequences.
For a sequence $m_{s} = \{m_{i}\}_{i=1}^{\infty}$, we denote the $\ell^{2}$-norm by $\|m_{s}\|_{0}$, that is,
$\|m_{s}\|_{0}^{2} = \sum_{i=1}^{\infty}m_{i}^{2}$. The hyperrectangle and Sobolev space of sequence of order $\beta > 0$ and radius $R > 0$ are
the sets
\begin{align}\label{no21}
\Theta^{\beta}(R) = \left\{ m_{s}\in\ell^{2} \,: \,\,\, \sup_{i \geq 1}i^{1+2\beta}m_{i}^{2} \leq R \right\},
\end{align}
\begin{align}\label{no22}
S^{\beta}(R) = \left\{ m_{s}\in\ell^{2} \,:\,\,\, \sum_{i=1}^{\infty}i^{2\beta}m_{i}^{2} \leq R \right\}.
\end{align}
For a sequence $m_{s}=\{m_{i}\}_{i=1}^{\infty}$, the hyperrectangle norm and Sobolev norm can be defined by
\begin{align}\label{norm1}
\|m_{s}\|_{h,\beta}^{2} = \sup_{i\geq 1}i^{1+2\beta}m_{i}^{2}, \quad
\|m_{s}\|_{\beta}^{2} = \sum_{i=1}^{\infty}i^{2\beta}m_{i}^{2}.
\end{align}

\begin{definition}\label{selfSimilarDef}
For some fixed positive constants $\delta$, $N_{0}$ and $\rho \geq 2$, we say a 
sequence $m_{s} = \{m_{i}\}_{i=1}^{\infty} \in \Theta^{\beta}(R)$ is self-similar if it satisfies
\begin{align}\label{self1}
\sum_{i=N}^{\rho N}m_{i}^{2} \geq \delta R N^{-2\beta}, \quad \forall N \geq N_{0}.
\end{align}
The class of self-similar elements of $\Theta^{\beta}(R)$ are denoted by $\Theta_{ss}^{\beta}(R, \delta)$.
The parameters $N_{0}$ and $\rho$ are fixed and omitted from the notation.
\end{definition}
This definition is employed by Gin\'{e} and Nickl \cite{Gine2010AS} and Bull \cite{Bull2012EJS} to remove a ``small'' set of undesirable truths
from the model, which allows to generate candidate confidence sets for the true parameter that are routinely used in practice.
In \cite{Szabo2015AS}, the authors propose the polished tail condition which includes the set of self-similar sequences.
The set of self-similar sequences has been shown natural once one has adopted the Bayesian setup with variable regularity prior probability measures.

\subsection{Assumptions}
In this section, we give the main assumptions employed in our work.
We assume that the prior probability measure and the probability measure of the noise $\xi$ have the following form
\begin{align*}
\mu^{0} = \mathcal{N}(0,\mathcal{C}^{\alpha}), \quad \mathbb{P}_{\text{noise}} = \mathcal{N}(0,\Sigma),
\end{align*}
where $\mathcal{C}:H\rightarrow H$ is a self-adjoint, positive-definite, trace class, linear operator and
$\Sigma:H\rightarrow H$ is assumed to be a self-adjoint, positive-definite linear operator.
The operator $\mathcal{A}:\mathcal{D}(\mathcal{A})\rightarrow H$ is assumed to be a self-adjoint and positive-definite, linear operator
with bounded inverse, $\mathcal{A}^{-1}:H\rightarrow H$.

Let us firstly present the following assumptions which describe the relations between Hilbert space and the space of sequences.

\textbf{Assumptions 1}: \label{Assumption2}
The covariance operator $\mathcal{C}$ has eigenpairs $\{\lambda_{i}^{2}, \phi_{i}\}_{i=1}^{\infty}$ on
Hilbert space $H$. For the singular values $\{\lambda_{i}\}_{i=1}^{\infty}$, there exist two positive constants $p > 0$ and $h > 0$ such that
\begin{align}\label{lan1}
\lambda_{i}^{2} \asymp i^{-\frac{2p}{h}}, \quad \forall \,\, i = 1,2,\ldots.
\end{align}
Formula (\ref{lan1}) means that there exist a series of constants $\{c_{i}\}_{i=1}^{\infty}$ and a positive constant $C > 0$ such that
$\lambda_{i}^{2} = c_{i}^{2}i^{-2p/h}$ with $C^{-1} \leq c_{i} \leq C$. Let $\alpha_{0}:= \frac{h}{2p}$, for all $\alpha > \alpha_{0}$, we have
\begin{align}\label{lan2}
\sum_{i = 1}^{\infty} i^{-\frac{2p\alpha}{h}} < \infty.
\end{align}
Inequality (\ref{lan2}) reflects that the operator $\mathcal{C}^{\alpha}$ is a trace class operator for $\alpha > \alpha_{0}$.
\begin{remark}
For concrete examples, the separable Hilbert space $H$ in Assumptions 1 is usually taken to be some function space defined on a 
$h$-dimensional closed manifold.
For more detailed illustrations, we refer to Section \ref{ExampleSection}.
\end{remark}

Secondly, we give the assumptions which are mainly concerned with the interrelations between the three operators $\mathcal{C}$,
$\Sigma$ and $\mathcal{A}^{-1}$. Similar to the assumptions presented in \cite{Agapiou2013IP}, these assumptions reflect
the ideas that
\begin{align}\label{assLoss1}
\Sigma \simeq \mathcal{C}^{\beta}, \quad \mathcal{A}^{-1} \simeq \mathcal{C}^{\ell},
\end{align}
for some $\beta \geq 0$, $\ell \geq 0$, where $\simeq$ is used loosely to indicate two operators which induce equivalent norms.
Specifically speaking, we make the following assumptions.

\textbf{Assumptions 2}: \label{Assumption1}
Let us denote $\mathcal{T} := \Sigma^{-1/2}\mathcal{A}^{-1}$ and $\mathcal{M}(\alpha) := \mathcal{T}\mathcal{C}^{\alpha}\mathcal{T}^{*}$.
Suppose there exist $\beta \geq 0$, $\ell \geq 0$, $\epsilon > 0$ and for all $\alpha > \alpha_{0}$ such that
\begin{enumerate}
  \item $\Delta \geq 1$, where $\Delta := 2\ell - \beta + 1$;
  \vskip 0.1cm
  \item $\big\|\mathcal{T}u\big\|_{\mathcal{H}^{r}} \asymp
  \big\|\mathcal{C}^{\frac{1}{2}(\Delta - 1)}u\big\|_{\mathcal{H}^{r}}, \quad \forall \,\, u\in \mathcal{H}^{r-(\Delta-1)}, \,\, r\geq 0$;
  \vskip 0.1cm
  \item $\big\|\mathcal{C}^{\frac{\alpha}{2}}\mathcal{T}^{*}u\big\| \lesssim
  \big\|u\big\|_{\mathcal{H}^{-\alpha-(\Delta-1)}}, \quad \forall \,\, u\in \mathcal{H}^{-\alpha-(\Delta-1)}$;
  \vskip 0.1cm
  \item $\big\| \mathcal{M}(\alpha)u \big\|_{\mathcal{H}^{2(\alpha+\Delta-1)}} \lesssim \big\|u\big\|, \quad \forall \,\, u\in H$;
  \vskip 0.1cm
  \item $\Big\| \mathcal{C}^{-\frac{1}{2}(\alpha+\Delta-1)+\frac{s}{2}}\mathcal{M}(\alpha)^{\frac{1}{2}}\mathcal{C}^{-\frac{s}{2}} \Big\|_{\mathcal{B}(H)} < \infty$;
  \vskip 0.1cm
  \item $\Big\| \mathcal{M}(\alpha)^{x}\mathcal{C}^{-x(\alpha+\Delta-1)} \Big\|_{\mathcal{B}(H)} < \infty$, \quad $\forall\, x\in[-\frac{1}{2},1]$.
\end{enumerate}

\begin{remark}\label{remarkAssump2}
If the forward operator $\mathcal{A}^{-1}$ and the noise covariance operator $\Sigma$ can be diagonalized under the eigensystem proposed in 
Assumptions 1, we may specify $\mathcal{A}^{-1}\phi_i = \zeta_i^{2}\phi_i$ and $\Sigma \phi_i = \xi_i^{2} \phi_i$ with 
$\lambda_i^{\ell} \lesssim \zeta_i \lesssim \lambda_i^{\ell}$ and $\lambda_i^{\beta}\lesssim \xi_i \lesssim \lambda_i^{\beta}$. 
Under this setting, we obviously have 
\begin{align}\label{diagonalT}
\mathcal{T}u = \sum_{i=1}^{\infty}\zeta_i^2\xi_i^{-1}u_i\phi_i. 
\end{align}
Let $C$ be a constant independent of $u$, then we have 
\begin{align}\label{remarkAss2_1}
C^{-1} \sum_{i=1}^{\infty}\lambda_i^{-2r+2(2\ell-\beta)}u_i^2 \leq \sum_{i=1}^{\infty}\lambda_i^{-2r}\zeta_i^{4}\xi_i^{-2}u_i^2 \leq C \sum_{i=1}^{\infty}\lambda_i^{-2r+2(2\ell-\beta)}u_i^2.
\end{align}
From (\ref{diagonalT}) and inequality (\ref{remarkAss2_1}), we can verify condition (2) in Assumptions 2. 
Noticing 
\begin{align*}
\mathcal{C}^{\alpha}\mathcal{T}^{*}u = \sum_{i=1}^{\infty}\lambda_i^{2\alpha}\zeta_i^2\xi_i^{-1}u_i\phi_i, \qquad
\mathcal{M}(\alpha)u = \sum_{i=1}^{\infty}\lambda_i^{2\alpha}\zeta_i^4 \xi_i^{-2}u_i\phi_i,
\end{align*} 
and employing the relations between $\lambda_i$, $\zeta_i$, and $\xi_i$ ($i=1,2,\cdots$), 
we find that conditions (3) and (4) hold true.
For conditions (5) and (6), there are no new operators appeared. 
Hence, we will not check the two conditions in detail for the current diagonal case. 
From these brief analysis, we obtain that the assumptions given in Assumptions 2 naturally hold for simultaneous diagonal case
and these assumptions indeed reflect the intuitive ideas illustrated in (\ref{assLoss1}). 
\end{remark}

At last, we assume that the truth $u^{\dag}$ belongs to $\mathcal{H}^{\gamma}$ with $\gamma \geq 1$.
Denote $u_{i}^{\dag} := (u^{\dag},\phi_{i})$ for $i = 1,2,\ldots$, and $u_{s}^{\dag} = \{u_{i}^{\dag}\}_{i=1}^{\infty}$.
Since
\begin{align}\label{truthReg1}
\|u^{\dag}\|_{\mathcal{H}^{\gamma}}^{2} = \sum_{i=1}^{\infty}\lambda_{i}^{-2\gamma}(u_{i}^{\dag})^{2}
\asymp \sum_{i=1}^{\infty}i^{\frac{2p}{h}\gamma}(u_{i}^{\dag})^{2},
\end{align}
we know that $u_{s}^{\dag} \in S^{\frac{p\gamma}{h}}(R)$ for some $R>0$.
In addition, we easily find that $u_{s}^{\dag} \in \Theta^{\frac{p\gamma}{h}}(R')$ for some $R' > 0$.
Concerning the truth, we make the following assumption which is crucial for our estimations.

\textbf{Assumption 3}: \label{Assumption3}
Let $u^{\dag} \in \mathcal{H}^{\gamma}$ for some $\gamma \geq 1$. In addition, we assume that
the sequence $m_{s}^{\dag}$ induced by $m^{\dag}$($m^{\dag} = \mathcal{T}u^{\dag}$) belongs to
$\Theta_{ss}^{\tilde{\beta}}(R, \delta)$ for some $\delta > 0$ and $R > 0$
with $\tilde{\beta} = \frac{p}{h}(\gamma+\Delta-1)$.

\section{Posterior contraction}\label{PosteriorCon}

In this section, we will present our main results without proof details to
give a sketch of our main ideas.
The proof details are all postponed to the Appendix.

As stated in the introduction, we will introduce two auxiliary problems which are crucial for our analysis.
Hence, it is necessary to give a clear summarization of the three problems employed in our work.

\textbf{Original Problem}: Under the assumption $\mathcal{R}(\mathcal{A}^{-1}) \subset \mathcal{D}(\Sigma^{-1/2})$,
we summarize the essential elements of the original inverse problem as follows:
\begin{itemize}
  \item Forward operator: $\mathcal{T} = \Sigma^{-1/2}\mathcal{A}^{-1}$,
  \item Data: $d = \mathcal{T}u + \frac{1}{\sqrt{n}}\eta$,
  \item Prior probability measure: $\mu^{0} = \mathcal{N}(0,\mathcal{C}^{\alpha})$ with $\alpha > \alpha_{0}$,
  \item Noise probability measure: $\mathbb{P}_{\text{noise}} = \mathcal{N}(0,I)$.
\end{itemize}

Since the operators $\mathcal{T}$ and $\mathcal{C}^{\alpha}$ cannot be diagonalized simultaneously, we introduce $m = \mathcal{T}u$
that transforms the forward operator to the identity operator. Then we propose the following transformed problem.

\textbf{Transformed Problem}:
For the transformed inverse problem, the necessary elements can be summarized as follows:
\begin{itemize}
  \item Forward operator: $I$,
  \item Data: $d = m + \frac{1}{\sqrt{n}}\eta$,
  \item Prior probability measure: $\mu^{m0} = \mathcal{N}(0,\mathcal{T}\mathcal{C}^{\alpha}\mathcal{T}^{*})$ with $\alpha > \alpha_{0}$,
  \item Noise probability measure: $\mathbb{P}_{\text{noise}} = \mathcal{N}(0,I)$.
\end{itemize}

The covariance operator $\mathcal{T}\mathcal{C}^{\alpha}\mathcal{T}^{*}$ of the transformed problem cannot be diagonalized
under the eigenbasis presented in Assumptions 1.
Hence, we can hardly obtain useful estimations of the corresponding eigenvalues, which inspired us to introduce an artificial diagonal problem.
For convenience, let us define 
\begin{align}
\mathcal{C}_{0} = \sum_{i=1}^{\infty}i^{-\frac{2p}{h}}\phi_i\otimes\phi_i,
\end{align}
where $\{\phi_i\}_{i=1}^{\infty}$ are the eigenvectors introduced in Assumptions 1
and $\otimes$ is the tensor product on Hilbert space \cite{Hsing2015book}. 
This definition naturally yields $\mathcal{C}_{0}\phi_{i} = i^{-\frac{2p}{h}}\phi_i$ for $i=1,2,\cdots$. 
Now, we are ready to present the following artificial diagonal problem, which is the starting point of our theoretical work. 

\textbf{Artificial Diagonal Problem}:
For the artificial diagonal problem, the essential elements can be summarized as follows:
\begin{itemize}
  \item Forward operator: $I$,
  \item Data: $d = m + \frac{1}{\sqrt{n}}\eta$,
  \item Prior probability measure: $\mu^{md0} = \mathcal{N}(0,\mathcal{C}_{0}^{\Delta-1+\alpha})$ with $\alpha > \alpha_{0}$,
  \item Noise probability measure: $\mathbb{P}_{\text{noise}} = \mathcal{N}(0,I)$.
\end{itemize}

We will recast the artificial diagonal problem into the framework introduced in \cite{Knapik2016, Szabo2015AS} and
provide some detailed analysis. Then relations of the artificial diagonal problem and the transformed problem will be explored to
provide a posterior contraction estimation for the transformed problem. At last, the general approach developed in \cite{Salomond2018B} will be
employed to transfer the posterior contraction estimation to our original problem.

\subsection{Analysis of the artificial diagonal problem}
Denote $m_{i} := (m,\phi_{i})$, $d_{i}:=(d,\phi_{i})$, and $\eta_{i}:=(\eta,\phi_{i})$, 
equality (\ref{S15}) can be rewritten as follows:
\begin{align}\label{s2a2}
d_{i} = m_{i} + \frac{1}{\sqrt{n}}\eta_{i}, \quad \text{for }i = 1,2,\ldots
\end{align}
with $m_{i} \sim \mathcal{N}(0,i^{-\frac{2p}{h}(\Delta-1+\alpha)})$ and $\eta_{i} \sim \mathcal{N}(0,1)$.
For convenience, let us define
\begin{align}\label{s2a3}
\tilde{\alpha} := \frac{p}{h}(\Delta - 1 + \alpha) - \frac{1}{2}.
\end{align}
Since $\Delta \geq 1$ and $\alpha > \alpha_{0}$, we find $\tilde{\alpha} \geq 0$. From formula (\ref{s2a3}), we easily know that $i^{-1-2\tilde{\alpha}} = i^{-\frac{2p}{h}(\Delta-1+\alpha)}$.
Following formula (2.2) employed in \cite{Knapik2016}, we introduce log-likelihood for $\tilde{\alpha}$ (relative to an infinite product of
$\mathcal{N}(0,1/n)$-distribution) as follows:
\begin{align}\label{s2a4}
\ell_{n}(\tilde{\alpha}) = -\frac{1}{2}\sum_{i=1}^{\infty}\left(
\log\left( 1+\frac{n}{i^{1+2\tilde{\alpha}}} \right) - \frac{n^{2}}{i^{1+2\tilde{\alpha}}+n}d_{i}^{2}
\right).
\end{align}

For the artificial diagonal problem, let $m^{\dag}$ be the truth which is equal to $\mathcal{T}u^{\dag}$, and denote
$m_{s}^{\dag} = \{m^{\dag}_{i}\}_{i=1}^{\infty} = \{(m^{\dag}, \phi_{i})\}_{i=1}^{\infty}$.
Considering $m^{\dag} = \mathcal{T}u^{\dag}$ and
\begin{align*}
\|m^{\dag}\|_{\mathcal{H}^{\gamma+\Delta-1}} = \|\mathcal{T}u^{\dag}\|_{\mathcal{H}^{\gamma+\Delta-1}}
\asymp \|\mathcal{C}^{\frac{1}{2}(\Delta-1)}u^{\dag}\|_{\mathcal{H}^{\gamma+\Delta-1}} = \|u^{\dag}\|_{\mathcal{H}^{\gamma}} < \infty,
\end{align*}
we obtain $m^{\dag} \in \mathcal{H}^{\gamma+\Delta-1}$. Relying on Assumptions 1, we have
\begin{align}\label{aa1}
\begin{split}
\|m^{\dag}\|_{\mathcal{H}^{\gamma+\Delta-1}}^{2} = & \sum_{i = 1}^{\infty} \lambda_{i}^{-2(\gamma+\Delta-1)}(m_{i}^{\dag})^{2} \\
= & \sum_{i = 1}^{\infty} c_{i}^{-2(\gamma+\Delta-1)}i^{\frac{2p}{h}(\gamma+\Delta-1)}(m_{i}^{\dag})^{2}.
\end{split}
\end{align}
The above equality (\ref{aa1}) indicates
\begin{align}\label{s2a5}
\|m_{s}^{\dag}\|_{\tilde{\beta}} \lesssim \|m^{\dag}\|_{\mathcal{H}^{\gamma+\Delta-1}} \lesssim \|m_{s}^{\dag}\|_{\tilde{\beta}}
\end{align}
with $\tilde{\beta} = \frac{p}{h}(\gamma+\Delta-1)$.
With these preparations, we define
\begin{align}\label{allhn1}
h_{n}(\tilde{\alpha}) = \frac{1+2\tilde{\alpha}}{n^{1/(1+2\tilde{\alpha})}\log n}
\sum_{i=1}^{\infty}\frac{n^{2}i^{1+2\tilde{\alpha}}(\log i)d_i^{2}}{(i^{1+2\tilde{\alpha}}+n)^{2}}
\end{align}
and
\begin{align}
\underline{\tilde{\alpha}}_{n} & = \inf \left\{
\tilde{\alpha} > 0 \,:\, h_{n}(\tilde{\alpha}) > \ell
\right\}\wedge \sqrt{\log n},     \label{allhn2}     \\
\overline{\tilde{\alpha}}_{n} & = \inf \left\{
\tilde{\alpha} > 0 \,:\, h_{n}(\tilde{\alpha}) > L \, (\log n)^{2}
\right\},  \label{allhn3}
\end{align}
where $0 < \ell < L < \infty$. As usual, the infimum of the empty set is considered $\infty$.
Before going further, let us present some estimates about $\underline{\tilde{\alpha}}_{n}$ and $\overline{\tilde{\alpha}}_{n}$.
\begin{lemma}\label{lemmaUpDownEst1}
For constants $R > 0$ and $\delta > 0$, there exist $C_{0}$ and $C_{1}$ such that
\begin{align}
\inf_{m_{s}^{\dag}\in S^{\tilde{\beta}}(R)} \underline{\tilde{\alpha}}_{n} \geq & \tilde{\beta} - C_{0}/\log n,     \label{formulaUpDownEst1} \\
\sup_{m_{s}^{\dag}\in \Theta_{ss}^{\tilde{\beta}}(R, \delta)\cap S^{\tilde{\beta}}(R)} \overline{\tilde{\alpha}}_{n}
\leq & \tilde{\beta} + C_{1}(\log\log n)/\log n,   \label{formulaUpDownEst2}
\end{align}
for $n$ large enough.
\end{lemma}
Then we define
\begin{align}\label{s2ax}
\hat{\tilde{\alpha}}_{n} := \mathop{\arg\max}_{\tilde{\alpha} \in [0,\log n]}\ell_{n}(\tilde{\alpha}) - \frac{C_{1}\log\log n}{\log n},
\end{align}
where $C_{1}$ is a positive constant appeared in Lemma \ref{lemmaUpDownEst1}.
Here, the constant $C_1$ depends on $\gamma$ that is usually not known in advance. However, a rough upper bound of $\gamma$
can be estimated in some practical applications, which is enough to give an upper bound of the constant. All of the illustrations
are hold when we replace $C_1$ to be some larger constant. 
Through some small modifications of the proof of Theorem 1 presented in \cite{Knapik2016}, we have
\begin{align}\label{allhn4}
\inf_{m_{s}^{\dag}\in S^{\tilde{\beta}}(R)\cap\Theta_{ss}^{\tilde{\beta}}(R,\delta)}
\mathbb{P}_{0}\left( \underline{\tilde{\alpha}}_{n} - C_{1}\frac{\log\log n}{\log n} \leq \hat{\tilde{\alpha}}_{n}
\leq \overline{\tilde{\alpha}}_{n} - C_{1}\frac{\log\log n}{\log n} \right)
\rightarrow 1
\end{align}
for any small constant $\delta>0$.
Obviously, when $\hat{\tilde{\alpha}}_{n}$ restricted to the following interval
\begin{align}\label{interval1}
\tilde{I}_{n}:=\left[ \underline{\tilde{\alpha}}_{n} - C_{1}\frac{\log\log n}{\log n}, \overline{\tilde{\alpha}}_{n} - C_{1}\frac{\log\log n}{\log n} \right],
\end{align}
we have
\begin{align}\label{estAlp1}
\hat{\tilde{\alpha}}_{n} \leq \tilde{\beta}  \quad \text{and} \quad
\hat{\tilde{\alpha}}_{n} \geq \tilde{\beta} - \frac{C_{2}\log\log n}{\log n},
\end{align}
where $C_{2} = C_{0} + C_{1}$. The notation $C_{2}$ will be used in all of the sections below.

Now, we give the following theorem which is critical for obtaining the posterior contraction of the transformed problem.
\begin{theorem}\label{converTheo1}
	Let $\hat{m}_{dn}(\tilde{\alpha})$ be the posterior mean estimator for our artificial diagonal problem when the regularity index is $\tilde{\alpha}$,
	then we have
	\begin{align}\label{convFor1}
	\sup_{m_{s}^{\dag}\in S^{\tilde{\beta}}(R)\cap \Theta_{ss}^{\tilde{\beta}}(R, \delta)}\mathbb{E}_{0}\left\{
	\sup_{\tilde{\alpha}\in \tilde{I}_{n}} \|\hat{m}_{dn}(\tilde{\alpha}) - m^{\dag}\|^{2}
	\right\} = O(\epsilon_{n}^{2}),
	\end{align}
	where $\epsilon_{n} = L_{n}n^{-\tilde{\beta}/(1+2\tilde{\beta})}$ and $L_{n}=(\log n)^{C_{2}+1}$.
\end{theorem}
\begin{remark}
In the Appendix, we provide a proof of Theorem \ref{converTheo1}, which is modified from the proof provided in Subsection 6.1 in \cite{Knapik2016}. 
For the current proof, the condition $m^{\dag}_s \in \Theta_{ss}^{\tilde{\beta}}(R,\delta)$ has indeed been used. However, for the artificial diagonal
problem, this condition can be removed if we employ more advanced proof techniques provided in \cite{Knapik2016, Szabo2015AS}.
For proving Theorem \ref{converTheo2} of the transformed problem, we need this condition to ensure an upper bound of $\hat{\tilde{\alpha}}_n$. 
Removing this condition here has little benefit for our original problem (proofs will become much more complicated), 
so we are satisfied with the results given in Theorem \ref{converTheo1}.
\end{remark}

As a preparation for the analysis of the transformed problem, we define
\begin{align}\label{gujiAl1}
\hat{\alpha}_{n} := \frac{h}{p}\left( \hat{\tilde{\alpha}}_{n} + \frac{1}{2} \right) + 1 - \Delta.
\end{align}
Considering formula (\ref{s2a3}), we know that $\hat{\alpha}_{n}$ above is the estimation for the regularity index of the prior $\mathcal{N}(0,\mathcal{C}^{\alpha})$ of $u$.
From $\hat{\tilde{\alpha}}_{n} \in \tilde{I}_{n}$, we easily deduce 
\begin{align}\label{estAlp2}
\begin{split}
\hat{\alpha}_{n} \leq \gamma + \frac{h}{2p}, \quad \text{and} \quad
\hat{\alpha}_{n} \geq \gamma + \frac{h}{2p} - \frac{hC_{2}\log\log n}{p\log n}.
\end{split}
\end{align}
In addition, based on (\ref{interval1}) and (\ref{gujiAl1}), we introduce 
\begin{align*}
I_n := \left[ \frac{h}{p}\left( \underline{\tilde{\alpha}}_{n} - \frac{C_{1}\log\log n}{\log n} + \frac{1}{2} \right) + 1 - \Delta,
\frac{h}{p}\left( \overline{\tilde{\alpha}}_{n} - \frac{C_{1}\log\log n}{\log n} + \frac{1}{2} \right) + 1 - \Delta \right].
\end{align*}
Obviously, we have $\alpha \in I_n$ if and only if $\tilde{\alpha}\in\tilde{I}_n$. 

\subsection{Posterior contraction for the transformed problem}\label{tranSection}
Denote $\mathcal{M}(\alpha) = \mathcal{T}\mathcal{C}^{\alpha}\mathcal{T}^{*}$ for $\alpha > \alpha_{0}$.
Now let us come back to the following transformed problem
\begin{align}\label{problem2}
d = m + \frac{1}{\sqrt{n}}\eta,
\end{align}
where
\begin{align}\label{condProblem2}
\eta \sim \mathcal{N}(0,I)  \quad \text{and} \quad
m \sim \mu^{m0} = \mathcal{N}(0,\mathcal{M}(\hat{\alpha}_{n}))
\end{align}
with $m = \mathcal{T}u$.
Before diving into the discussions on posterior contraction, let us provide some important estimates. Consider the equation
\begin{align}\label{equation1}
(n\mathcal{M}(\alpha) + I)u = r.
\end{align}
Define the bilinear form $B: H\times H\rightarrow \mathbb{R}$,
\begin{align}\label{binlinear1}
B(u,v) := (n^{1/2}\mathcal{C}^{\alpha/2}\mathcal{T}^{*}u, n^{1/2}\mathcal{C}^{\alpha/2}\mathcal{T}^{*}v) + (u,v), \quad \forall\, u,v\in H.
\end{align}
\begin{definition}\label{weakSolution1}
Let $r\in H$. An element $u\in H$ is called a weak solution of (\ref{equation1}), if
\begin{align*}
B(u,v) = (r,v), \quad \forall \, v\in H.
\end{align*}
\end{definition}
\begin{lemma}\label{weakSolLemma}
Under Assumptions 2, for any $r\in H$, there exists a unique weak solution $u \in H$ of (\ref{equation1}).
In addition, if $r \in \mathcal{H}^{t}$ with $t \leq 2(\alpha+\Delta-1)$, the weak solution $u \in \mathcal{H}^{t}$.
\end{lemma}
Relying on the definition of weak solution and
interpolation inequalities, we can prove the following lemma.
\begin{lemma}\label{opEstLemma1}
For any constant $s \in (\alpha_{0}, \alpha)$, under Assumptions 2, the following norm bound holds:
\begin{align}\label{opEst1}
\|\mathcal{C}^{\frac{1}{2}(\alpha+\Delta-1)-\frac{s}{2}}(n\mathcal{M}(\alpha)+I)^{-1}\mathcal{C}^{\frac{1}{2}(\alpha+\Delta-1)-\frac{s}{2}}\|_{\mathcal{B}(H)}
\lesssim n^{-1+\frac{s}{\alpha+\Delta-1}}.
\end{align}
\end{lemma}
According to the definitions of $\mathcal{C}_0$ and $\mathcal{C}$, especially Assumptions 1, we have the following result which 
links the estimates concerned with $\mathcal{C}$ and $\mathcal{C}_0$. 
\begin{lemma}\label{C0C}
For any constant $s\in\mathbb{R}$ and any $u\in H$, the following estimate holds:
\begin{align}
\|\mathcal{C}^{s} u\| \asymp \|\mathcal{C}_{0}^{s}u\|.
\end{align}
\end{lemma}
Let $\mu_{\hat{\alpha}_{n}}^{d}$ be the posterior probability measure for our original problem and
$\mu_{\hat{\alpha}_{n}}^{md}$ be the posterior probability measure for the transformed problem.
The two posterior probability measures are all Gaussian due to our assumptions.
Using the relations between $m$ and $u$, we know that
\begin{align}\label{zhongyao1}
\mu_{\hat{\alpha}_{n}}^{d}\left\{ u \,: \, \|\mathcal{T}u - \mathcal{T}u^{\dag}\| \geq M_{n}\tilde{\epsilon}_{n} \right\}
= \mu_{\hat{\alpha}_{n}}^{md}\left\{ m \,: \, \|m - m^{\dag}\| \geq M_{n}\tilde{\epsilon}_{n} \right\} ,
\end{align}
where $M_{n}$ and $\tilde{\epsilon}_{n}$ are positive sequences satisfy $M_{n}\rightarrow \infty$ and $\tilde{\epsilon}_{n}\downarrow 0$, respectively.
Let $\mathbb{E}_{\alpha,n}$ and $\hat{m}_{n}(\alpha)$ denote the expectation and the conditional mean estimator with respect to
the posterior probability measure $\mu_{\alpha}^{md}$, respectively.
Let $\mathbb{E}_{c,\alpha}$ be the expectation with respect to a probability measure $\mu_{c,\alpha}^{md}$ with zero mean and the same covariance operator as the posterior probability measure $\mu_{\alpha}^{md}$ of the transformed problem.
Let $w_{n}$ be a Gaussian random variable that is independent of the noise and distributed according to $\mu_{c,\alpha}^{md}$. 
Obviously, $w_{n}$ and $\hat{m}_{n}(\alpha)$ are independent for fixed $\alpha$.
Relying on the definition of $w_n$, we have 
\begin{align}\label{zhongyao2}
\mu_{\hat{\alpha}_{n}}^{md}\!\left\{ m :  \|m - m^{\dag}\| \geq M_{n}\tilde{\epsilon}_{n} \right\} \! = \!
\mu_{c, \hat{\alpha}_{n}}^{md}\!\left\{ w_n : \! \|w_n + \hat{m}_{n}(\alpha) - m^{\dag}\| \geq M_{n}\tilde{\epsilon}_{n} \right\}.
\end{align}
Combining (\ref{zhongyao1}), (\ref{zhongyao2}) and employing Markov's inequality, we obtain
\begin{align}\label{MarkovTrans1}
\begin{split}
& \sup_{m_{s}^{\dag}\in S^{\tilde{\beta}}(R)\cap \Theta_{ss}^{\tilde{\beta}}(R,\delta)}
\mathbb{E}_{0}\mu_{\hat{\alpha}_{n}}^{d}\left\{ u \,: \, \|\mathcal{T}u - \mathcal{T}u^{\dag}\| \geq M_{n}\tilde{\epsilon}_{n} \right\} \\
& \,\,\,
\leq \frac{1}{M_{n}^{2}\tilde{\epsilon}_{n}^{2}}\sup_{m_{s}^{\dag}\in S^{\tilde{\beta}}(R)\cap \Theta_{ss}^{\tilde{\beta}}(R,\delta)}\mathbb{E}_{0}
\sup_{\alpha\in I_{n}}
\mathbb{E}_{c, \alpha}\left( \|w_n + \hat{m}_{n}(\alpha) - m^{\dag}\|^{2} \right) + o(1).
\end{split}
\end{align}
Then, we can insert the posterior mean estimator of our artificial diagonal problem into the righthand side of (\ref{MarkovTrans1}).
Relying on Assumptions 2 and some calculations, we can prove the following theorem.
\begin{theorem}\label{converTheo2}
For every $R > 0, \delta>0$, $M_n\rightarrow\infty$, and an arbitrarily small positive constant $\tilde{\epsilon} > 0$, we have
\begin{align}
\sup_{m_{s}^{\dag}\in S^{\tilde{\beta}}(R)\cap \Theta_{ss}^{\tilde{\beta}}(R,\delta)}
\mathbb{E}_{0}\mu_{\hat{\alpha}_{n}}^{d}\left\{ u \,: \, \|\mathcal{T}u - \mathcal{T}u^{\dag}\| \geq M_{n}\tilde{\epsilon}_{n} \right\} \rightarrow 0, \quad
\text{as}\,\,\, n\rightarrow \infty,
\end{align}
where $\tilde{\epsilon}_{n} = L_{n}n^{-\tilde{\beta}/(1+2\tilde{\beta}+\tilde{\epsilon})}$ and $L_{n}=(\log n)^{C_{2}+1}$.
\end{theorem}

The above theorem provides a contraction estimate for the transformed problem which links the artificial diagonal problem and
the original problem. Different to the simultaneous diagonalizable case, we confine $m_{s}^{\dag}$ to belong to
$\Theta_{ss}^{\tilde{\beta}}(R, \delta)$. This condition allows us to employ the upper bound shown in Lemma \ref{lemmaUpDownEst1}
which is crucial for obtaining an appropriate estimate of the following term
\begin{align*}
\mathbb{E}_0\sum_{\alpha\in I_n}\|\hat{m}_n(\alpha) - \hat{m}_{dn}(\alpha)\|^2,
\end{align*} 
where $\hat{m}_{dn}(\alpha)$ is the conditional mean estimator for the artificial diagonal problem. Details can be found in the Appendix.

\subsection{Posterior contraction for the original problem}

The final step is to obtain the posterior contraction estimation by employing Theorem 2.1 proved in \cite{Salomond2018B} which
provides a general framework for obtaining posterior contraction rate of linear inverse problems.
Denote $\mu_{\ell} = \mathcal{N}(\mathcal{A}^{-1}u^{\dag}, \frac{1}{n}\Sigma)$ and 
$\mu_{\ell}' = \mathcal{N}(\mathcal{A}^{-1}u, \frac{1}{n}\Sigma)$.
For $u^{\dag} \in \mathcal{H}^{\gamma}$ and $u\in H$, we know that $\mathcal{A}^{-1}u^{\dag}, \mathcal{A}^{-1}u \in \mathcal{H}^{2\ell}$.
The Cameron-Martin space for $\mu_{\ell}$ and $\mu_{\ell}'$ are all $\mathcal{H}^{\beta}$.
By Assumptions 2, we easily obtain $\mathcal{A}^{-1}u, \mathcal{A}^{-1}u^{\dag} \in \mathcal{H}^{\beta}$ which implies that
$\mu_{\ell}$ and $\mu_{\ell}'$ are equivalent. Inspired by the proof of Theorem 3.1 in \cite{Salomond2018B},
for given sequences of positive numbers $k_{n}\rightarrow\infty$ and $\rho_{n}\rightarrow 0$ and a constant $c\geq 0$, we introduce
\begin{align}\label{set1}
S_{n} = \Big\{ u \in H \,:\, \sum_{i \geq k_{n}}u_{i}^{2} \leq c\rho_{n}^{2} \Big\},
\end{align}
and
\begin{align}\label{set2}
\begin{split}
B_{n}(\mathcal{A}^{-1}u^{\dag},\epsilon) = & \Bigg\{ u \in H \,:\, -\int\log\frac{d\mu_{\ell}'}{d\mu_{\ell}}d\mu_{\ell} \leq n\epsilon^{2}, \\
& \qquad\quad
\int \left| \log\frac{d\mu_{\ell}'}{d\mu_{\ell}} - \int\log\frac{d\mu_{\ell}'}{d\mu_{\ell}}d\mu_{\ell} \right|^{2} d\mu_{\ell} \leq n\epsilon^{2} \Bigg\}.
\end{split}
\end{align}
Through some calculations presented in the Appendix, we obtain
\begin{align}\label{set3}
B_{n}(\mathcal{A}^{-1}u^{\dag},\epsilon) = \Big\{ u\in H \,:\, \|\mathcal{T}(u - u^{\dag})\|^{2} \leq \epsilon^{2} \Big\}.
\end{align}
The proof of Lemma 1 in \cite{Ghosal2007AS} can be adapted a little to achieve the following lemma.
\begin{lemma}\label{cond1lemma}
Let $\delta_{n} \rightarrow 0$ and let $S_{n}$ be a sequence of sets $S_{n} \subset H$.
If $\mu_{n}^{0} = \mathcal{N}(0,\mathcal{C}^{\hat{\alpha}_{n}})$ is the prior probability measure on $u$ satisfying
\begin{align*}
\frac{\mu_{n}^{0}(S_{n}^{c})}{\mu_{n}^{0}(B_{n}(\mathcal{A}^{-1}u^{\dag},\delta_{n}))} \lesssim \exp\left( -2n\delta_{n}^{2} \right),
\end{align*}
then
\begin{align*}
\mathbb{E}_{0}\mu_{\hat{\alpha}_{n}}^{d}(S_{n}^{c}) \rightarrow 0
\end{align*}
with $\mu_{\hat{\alpha}_{n}}^{d}$ represents the posterior probability measure of the original problem.
\end{lemma}
Relying on Lemma \ref{cond1lemma}, we finally obtain our main result as follows.
\begin{theorem}\label{mainTheorem}
For every $R > 0, \delta>0$, $M_n\rightarrow\infty$, $C_{3}:= \frac{hC_{2}}{h+2p(\Delta-1)}$ and an arbitrarily small positive constant
$\epsilon > 0$, we have
\begin{align}\label{mainThC1}
\sup_{m_{s}^{\dag}\in S^{\tilde{\beta}}(R)\cap \Theta_{ss}^{\tilde{\beta}}(R,\delta)}
\mathbb{E}_{0}\mu_{\hat{\alpha}_{n}}^{d}\left\{ u \,: \, \|u - u^{\dag}\| \geq M_{n}\epsilon_{n} \right\} \rightarrow 0, \quad
\text{as}\,\,\, n\rightarrow \infty,
\end{align}
where
\begin{align*}
\epsilon_{n} = \left( \log n \right)^{C_{2}+C_{3}+1}n^{\frac{-p\gamma}{h+2p(\gamma+\Delta-1)+\epsilon}}.
\end{align*}
\end{theorem}
\begin{remark}
Excluding the logarithmic terms, the obtained contraction rate can be rewritten as follows:
\begin{align}\label{rem1}
n^{-\frac{\frac{p}{h}\gamma}{1+2\frac{p}{h}(\gamma+\Delta-1)+\frac{\epsilon}{h}}}. 
\end{align}
According to \cite{Knapik2011AS,Salomond2018B,Agapiou2018S}, the optimal contraction rate of the artifical diagonal problem is $n^{-\frac{\tilde{\beta}}{1+2\tilde{\beta}}}$ that can be rewritten as follows:
\begin{align}\label{rem2}
n^{-\frac{\frac{p}{h}(\gamma+\Delta-1)}{1+2\frac{p}{h}(\gamma+\Delta-1)}}. 
\end{align}
Based on the analysis shown in \cite{Salomond2018B} (Theorem 3.1 in \cite{Salomond2018B}), we know that the optimal contraction rate should be
\begin{align}\label{rem3}
n^{-\frac{\frac{p}{h}\gamma}{1+2\frac{p}{h}(\gamma+\Delta-1)}},
\end{align}
when $\mathcal{A}^{-1}$ and $\Sigma$ can be diagonalized under the eigensystem proposed in Assumptions 1. 
Comparing (\ref{rem1}) with (\ref{rem3}), we conclude that Theorem \ref{mainTheorem} provides an almost optimal contraction rate up to some 
logarithmic terms. 
\end{remark}

\section{Examples}\label{ExampleSection}
In this section, we provide some nontrivial examples.
For simplicity, we consider a simple closed manifold that is $h$-dimensional torus denoted by $\mathbb{T}^{h}$.
Introduce the Hilbert space $H$ defined as follows:
\begin{align*}
H := \dot{L}^{2}(\mathbb{T}^{h}) = \Big\{
u \,:\, \mathbb{T}^{h} \rightarrow \mathbb{R} \,\, \Big| \,\, \int_{\mathbb{T}^{h}}|u(x)|^{2}dx < \infty, \,
\int_{\mathbb{T}^{h}}u(x)dx = 0 \Big\}
\end{align*}
of real valued periodic function $h \leq 3$ with inner-product and norm denoted by $(\cdot,\cdot)$ and $\|\cdot\|$ respectively.
Let $\mathcal{A}_{0} := -\Delta$ be the negative Laplacian equipped with periodic boundary condition on $[0,1)^{h}$, and
restricted to functions with integrate to zero over $[0,1)^{h}$.
It is well-known that this operator is positive self-adjoint and has eigensystem $\{\rho_{j}^{2}, \phi_{j}\}_{j = 1}^{\infty}$.
The eigenfunctions $\{\phi_{j}\}_{j=1}^{\infty}$ constitute the Fourier basis and form a complete orthonormal basis of $H$ and
the eigenvalues $\rho_{j}^{2}$ behave asymptotically like $j^{2/h}$.

Before going further, let us provide two basic definitions concerned with the pesudodifferential operators.
For further details, we refer to \cite{Kekkonen2016IP,Shubin1987book}.
\begin{definition}
Let $m\in\mathbb{R}$. Then $S^{m}(\mathbb{R}^{h}, \mathbb{R}^{h})$ is the vector-space of all smooth functions $a\in C^{\infty}(\mathbb{R}^{h},\mathbb{R}^{h})$
such that
\begin{align*}
|\partial_{\xi}^{\alpha}\partial_{x}^{\beta}a(x,\xi)| \leq C_{\alpha,\beta,K}(1+|\xi|)^{m-|\alpha|}, \quad \xi\in\mathbb{R}^{h}, \,\,\,x\in K,
\end{align*}
holds for all multi-indices $\alpha$ and $\beta$ and any compact set $K \subset \mathbb{R}^{h}$.
\end{definition}
\begin{definition}
Let $Y : U\rightarrow \mathbb{R}^{h}$ be local coordinates of the manifold $N$. A bounded linear operator $A : \mathcal{D}'(N) \rightarrow \mathcal{D}'(N)$
is called a pseudodifferential operator if for any local coordinates $Y : U \rightarrow \mathbb{R}^{h}$, $U\subset N$,
there is a symbol $a\in S^{m}(\mathbb{R}^{h},\mathbb{R}^{h})$ such that for $u\in C_{0}^{\infty}(U)$ we have
\begin{align*}
Au(y_{1}) = \int_{N}k_{A}(y_{1},y_{2})u(y_{2})dV_{g}(y_{2}),
\end{align*}
where $k_{A}|_{N\times N\backslash\text{diag}(N)}\in C^{\infty}(N\times N\backslash \text{diag}(N))$ and
$\text{diag}(N)=\{(y,y)\in N\times N | y\in N\}$. Also when $Y:U\rightarrow V\subset\mathbb{R}^{h}$ are local $C^{\infty}$-smooth coordinates
$k_{A}(y_{1},y_{2})$ is given on $U\times U$ by
\begin{align*}
k_{A}(Y^{-1}(x_{1}),Y^{-1}(x_{2})) = \int_{\mathbb{R}^{h}}e^{i(x_{1}-x_{2})\xi}a(x_{1},\xi)d\xi,
\end{align*}
where $x_{1},x_{2}\in V\subset \mathbb{R}^{h}$ and $a = a_{V} \in S^{m}(V,\mathbb{R}^{h})$.
In this case we will write
\begin{align*}
A \in \Psi^{m}(N),
\end{align*}
and say that in local coordinates $Y : U\rightarrow V\subset \mathbb{R}^{h}$ the operator $A$ has the symbol $a(x,\xi)\in S^{m}(V\times\mathbb{R}^{h})$.
\end{definition}

Following the proof of Lemma 1 in \cite{Kekkonen2016IP}, we can obtain the following lemma.
\begin{lemma}\label{boundPesudoLemma}
Let $B \in \Psi^{-2t}(\mathbb{T}^{h})$ be an injective elliptic pseudodifferential operator. Then we have the following estimates
\begin{align*}
\|Bu\|_{\dot{H}^{r+2t}(\mathbb{T}^{h})} \lesssim \|u\|_{\dot{H}^{r}(\mathbb{T}^{h})} \lesssim \|Bu\|_{\dot{H}^{r+2t}(\mathbb{T}^{h})},
\quad \forall \,\, r\in\mathbb{R}.
\end{align*}
\end{lemma}

\subsection{Example 1--non-diagonal forward operator}\label{Ex1SubSection}
Let $\mathcal{M}_{q}: H \rightarrow H$ be the multiplication operator by a nonnegative function $q \in C^{\infty}(\mathbb{T}^{h})$.
We define the forward operator $\mathcal{A}^{-1} := (\mathcal{A}_{0} + \mathcal{M}_{q})^{-1}$, assume that the observational noise is white, 
so that $\Sigma = I$, and we set the operator $\mathcal{C} = \mathcal{A}_{0}^{-2}$.

Denote the eigenvalues of the operator $\mathcal{C}$ to be $\{\lambda_{j}^{2}\}_{j=1}^{\infty}$.
Obviously, we have $\lambda_{j}^{2} = \rho_{j}^{-4}$. Since $\sum_{j=1}^{\infty}\lambda_{j}^{2} \asymp \sum_{j=1}^{\infty}j^{-\frac{4}{h}} < \infty$ for $h \leq 3$,
the operator $\mathcal{C}$ is trace class and the constant $p$ appeared in Assumptions 1 is equal to $2$.

Our aim is to show $\Sigma\simeq \mathcal{C}^{\beta}$ and $\mathcal{A}^{-1}\simeq \mathcal{C}^{\ell}$, where $\beta = 0$ and
$\ell = \frac{1}{2}$ in the sense specified in Assumptions 2. Obviously, we know $\Delta = 2\ell - \beta + 1 = 2$.
Under these settings, we easily obtain the equivalent relations of Hilbert scales and Sobolev space as follows:
\begin{align*}
\mathcal{H}^{t} = \dot{H}^{2t}(\mathbb{T}^{h}).
\end{align*}
Now, we verify all of the conditions appeared in Assumptions 2.
\begin{description}
  \item[(2)] For this condition, we need to prove the following statement
  \begin{align*}
  \|\mathcal{A}^{-1}u\|_{\dot{H}^{2r}(\mathbb{T}^d)} \asymp \|u\|_{\dot{H}^{2r-2}(\mathbb{T}^d)}, \quad \forall \,\, r\geq 0.
  \end{align*}
  The forward operator $\mathcal{A}^{-1}$ is a pseudodifferential operator with symbol
  $(|\xi|^{2} + q(x))^{-1}$. That is to say, $\mathcal{A}^{-1} \in \Psi^{-2}(\mathbb{T}^{h})$ and $\mathcal{A}^{-1}$ is an injective
  elliptic pseudodifferential operator. Through Lemma \ref{boundPesudoLemma}, we complete the verification.
  \item[(3)] The operator $\mathcal{C}^{\frac{\alpha}{2}}\mathcal{T}^{*} = \mathcal{C}^{\frac{\alpha}{2}}\mathcal{A}^{-1}$ is a pseudodifferential operator with symbol
  $|\xi|^{-2\alpha}(|\xi|^{2}+q(x))^{-1}$.
  Hence, we have $\mathcal{C}^{\frac{\alpha}{2}}\mathcal{T}^{*} \in \Psi^{-2(1+\alpha)}(\mathbb{T}^d)$
  which implies
  \begin{align*}
  \|\mathcal{C}^{\frac{\alpha}{2}}\mathcal{T}^{*}u\|_{\dot{L}^{2}(\mathbb{T}^{h})} \lesssim \|u\|_{\dot{H}^{-2(1+\alpha)}(\mathbb{T}^{h})}.
  \end{align*}
  \item[(4)] To verify this condition, let us notice that the operator
  $$\mathcal{C}^{-(\alpha+\Delta-1)}\mathcal{M}(\alpha) \in \Psi^{0}(\mathbb{T}^{h}),$$
  which indicates
  \begin{align*}
  \|\mathcal{C}^{-(\alpha+\Delta-1)}\mathcal{M}(\alpha)u\| \lesssim \|u\|_{\dot{L}^{2}(\mathbb{T}^{h})}, \quad \forall\,\, u\in \dot{L}^{2}(\mathbb{T}^{h}).
  \end{align*}
\end{description}
Conditions (5) and (6) can be verified similarly, here, we omit the details for concisely.
By choosing the regularity index according to (\ref{gujiAl1}), we can apply Theorem \ref{mainTheorem}
to obtain the following convergence result.
\begin{theorem}\label{Ex1Theorem}
Let $u^{\dag}\in \dot{H}^{2\gamma}(\mathbb{T}^{h})$, $\gamma \geq 1$. Then, the convergence in (\ref{mainThC1}) holds with
\begin{align}\label{Ex1Th1}
\tilde{\epsilon}_{n} = \left( \log n \right)^{C_{2}+C_{3}+1}n^{\frac{-2\gamma}{4+h+4\gamma+\epsilon}},
\end{align}
for some large enough constants $C_{2}$ and $C_{3}$ and arbitrarily small constant $\epsilon>0$.
\end{theorem}
Comparing with Theorem 8.2 proved in \cite{Agapiou2013IP}, we notice that this convergence rate is optimal up to a slowly varying factor
and an arbitrarily small positive constant $\epsilon > 0$.

\subsection{Example 2--a fully non-diagonal example}
As in Subsection \ref{Ex1SubSection}, let the forward operator to be $\mathcal{A}^{-1} = (\mathcal{A}_{0} + \mathcal{M}_{q})^{-1}$.
We assume that the observational noise is Gaussian with covariance operator $\Sigma = (\mathcal{A}_{0} + \mathcal{M}_{r})^{-2}$,
where $\mathcal{M}_{r}$ is the multiplication operator by a nonnegative function $r\in C^{\infty}(\mathbb{T}^h)$.
We also assume $\mathcal{C} = \mathcal{A}_{0}^{-2}$.

Under this setting, we would like to verify Assumptions 2 with $\ell = \frac{1}{2}$ and $\beta = 1$.
We easily know that $\mathcal{A}^{-1}\in \Psi^{-2}(\mathbb{T}^h)$, $\Sigma\in\Psi^{-4}(\mathbb{T}^h)$ and $\mathcal{C}^{\alpha}\in\Psi^{-4\alpha}(\mathbb{T}^h)$.
Assumptions 2 
can be verified through similar analysis shown in Subsection \ref{Ex1SubSection}.
Then we apply Theorem \ref{mainTheorem} to obtain the following convergence result.
\begin{theorem}\label{Ex2Theorem}
Let $u^{\dag}\in \dot{H}^{2\gamma}(\mathbb{T}^{h})$, $\gamma \geq 1$. Then, the convergence in (\ref{mainThC1}) holds with
\begin{align}\label{Ex2Th1}
\tilde{\epsilon}_{n} = \left( \log n \right)^{C_{2}+C_{3}+1}n^{\frac{-2\gamma}{h+4\gamma+\epsilon}},
\end{align}
for some large enough constants $C_{2}$ and $C_{3}$ and arbitrarily small constant $\epsilon > 0$.
\end{theorem}
\begin{remark}
For Example 2, we may transform the problem to the form of (\ref{S14}) with the white noise and operator 
$\mathcal{T} = (\mathcal{A}_0 + \mathcal{M}_r)\circ (\mathcal{A}_0 + \mathcal{M}_q)^{-1}$. Comparing the 
forward operator $(\mathcal{A}_0 + \mathcal{M}_q)^{-1}$ in Example 1 with the operator $\mathcal{T}$ in Example 2, we 
intuitively find that the degree of ill-posedness for Example 1 is larger than that for Example 2. 
Hence, it is predictable that the optimal posterior contraction rate for Example 2 will be faster than the one for Example 1. 
The results presented in Theorems \ref{Ex1Theorem} and \ref{Ex2Theorem} justify the above intuitive illustrations quantitatively and rigorously. 
\end{remark}

\section{Numerical illustration}\label{NumSection}
In this section, we plan to give some numerical results that illustrate the effectiveness of the proposed empirical Bayes' method.
For the high-dimensional problems, how to discretize the fractional Laplace operator
is an activate research problem \cite{Ilic2005FCAA}. Effective numerical implementation is not the main
focus of this paper, so we set the space dimension $h=1$ for simplicity.
Recall the basic settings in Subsection \ref{Ex1SubSection}, 
we know that $p=2$, $\beta=0$, $\ell=\frac{1}{2}$, and $\Delta=2$. 
In the following, we take $\gamma$ given in Theorem \ref{Ex1Theorem} equal to $1$. 
We choose the true function $u^{\dag}\in\mathcal{H}^{\gamma}$ with the following form
\begin{align}
u^{\dag} = \sum_{i=1}^{\infty}i^{-2\gamma - \frac{1}{2}-\epsilon}\phi_i,
\end{align}
where $\{\phi_i\}_{i=1}^{\infty}$ are the eigenbasis introduced in Assumptions 1 
and $\epsilon>0$ is a fixed small number. 

\begin{figure}[htbp]
	\centering
	\includegraphics[width=0.75\textwidth]{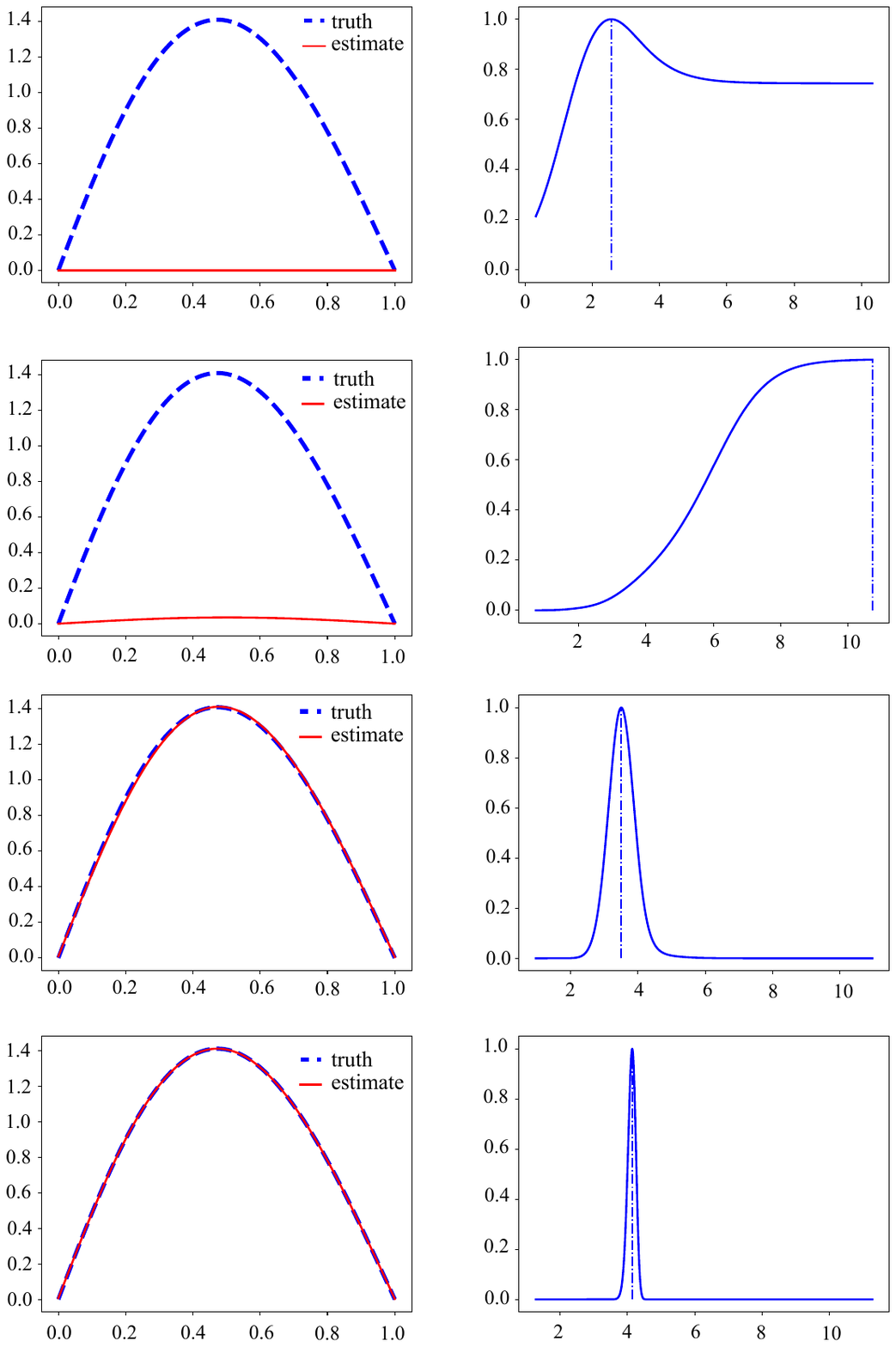}\\
	\vspace*{-0.6em}
	\caption{Left panels the empirical Bayes posterior mean (red) and the true curve (blue, dashed).
	Right panels corresponding normalized likelihood for $\hat{\tilde{\alpha}}$ (regularity index for the artificial diagonal problem). 
	We have $n = 10^3, 10^5, 10^8,$ and $10^{12}$, from top to bottom.}\label{Fig1}
\end{figure}

The forward operator $\mathcal{A}^{-1}$ and the covariance operator $\mathcal{C}$ are discretized by employing finite element method
implemented with the open source software FEniCS (Version 2019.1.0). The related fractional operator $\mathcal{C}^{\alpha}$ with $\alpha>0$
is discretized through the method introduced in Subsection 3.4.1 in \cite{Thanh2016IPI}.
With these specified parameters, we easily know that $\tilde{\beta}$ introduced for the artificial diagonal problem should be equal to $4$. 
That is to say, the estimated $\hat{\tilde{\alpha}}_n$ should around $4$ when $n$ is large. 
In Figure \ref{Fig1}, we plot the true function $u^{\dag}$ (black dashed curve) and the empirical Bayes mean (red curve) 
in the left panels, and the corresponding normalized likelihood $\exp(\ell_n)/\max(\exp(\ell_n))$ in the right panels 
(we truncated the sum in (\ref{allhn1}) at a high level). Here, we regardless of the term $C_1\log\log n/\log n$ since this term tends to $0$
when $n\rightarrow \infty$. The figure shows that the estimated $\hat{\tilde{\alpha}}_n$ does a good job in this case at estimating
the regularity level $4$, at least for large enough $n$. Through formula (\ref{gujiAl1}), the proposed method provide an accurate 
estimate for the regularity index of the original problem. 

\section{Conclusions}\label{ConSection}

In this paper, we study an empirical Bayesian approach to a family of linear inverse problems.
We assume that the covariance operator of the prior Gaussian measure depends on a hyperparameter that quantifies the regularity.
We do not assume the prior covariance, noise covariance and forward operator have a common basis in their singular value decomposition,
which enlarge the application range compared with the existing results.
Under such weak assumptions, it is difficult to introduce maximum likelihood based estimation for the regularity index.
In order to construct an empirical Bayesian approach, we propose two auxiliary problems: transformed problem and artificial diagonal problem.
Because the regularity index only reflects the information about regularity, we provide a maximum likelihood estimation of the regularity index through
the artificial diagonal problem. Then we deduce the posterior contraction estimates for the transformed problem by exploring the relations between
the transformed problem and the artificial diagonal problem.
Finally, the desired posterior contraction estimates are obtained through a general approach developed recently.


For the posterior contraction rate estimate problem, there are numerous challenging problems need to be solved.
Only inverse problems with linear forward operator have been considered in this paper, how to design
empirical Bayesian methods for inverse problems with nonlinear forward operator is a difficult problem and
it seems to depend on the specific structures of the certain problem.

\section{Appendix}\label{proofSection}

Here we gather all of the proof details.

\vskip 0.3cm
\textbf{Proof of Lemma \ref{lemmaUpDownEst1}}
\begin{proof}
The proof of estimation (\ref{formulaUpDownEst1}) is similar to the proof of (i) of Lemma 1 presented in \cite{Knapik2016},
so we omit the details. For estimation (\ref{formulaUpDownEst2}),
let us denote $\tilde{\beta}_{n}:=\tilde{\beta}+C_{1}(\log\log n)/\log n$.
Employing Lemma 6.1 proved in \cite{Szabo2015AS}, we find that
\begin{align}\label{proofLe1}
h_{n}(\tilde{\beta}_{n}) \geq \frac{\epsilon R}{(\rho^{1+2\tilde{\beta}_{n}}+1)^{2}}n^{2(\tilde{\beta}_{n}-\tilde{\beta})/(1+2\tilde{\beta}_{n})}
\geq \frac{\epsilon R}{(\rho^{2+2\tilde{\beta}}+1)^{2}}n^{2(\tilde{\beta}_{n}-\tilde{\beta})/(1+2\tilde{\beta}_{n})}
\end{align}
holds for choosing $n$ large enough such that $C_{1}(\log\log n)/\log n \leq 1/2$.
Because $(\log\log n)/\log n \leq 1/4$ for a large enough $n$ ($n\geq5600$), we have
\begin{align}\label{proofLe2}
n^{2(\tilde{\beta}_{n}-\tilde{\beta})/(1+2\tilde{\beta}_{n})} \geq n^{\frac{1}{\log n}(\log\log n)\frac{2C_{1}}{1+2\tilde{\beta}+C_{1}/2}}
= (\log n)^{2C_{1}/(1+2\tilde{\beta}+C_{1}/2)}.
\end{align}
Hence, taking $C_{1} > 6 + 12\tilde{\beta}$, we obtain
$h_{n}(\tilde{\beta}_{n}) \geq L (\log n)^{2}$ for large $n$, which verifies estimation (\ref{formulaUpDownEst2}).
\end{proof}

\textbf{Proof of Theorem \ref{converTheo1}}
\begin{proof}
Because
\begin{align*}
\|\hat{m}_{dn}(\tilde{\alpha}) - m^{\dag}\|^{2} = \sum_{i=1}^{\infty}(\hat{m}_{dn,i}(\tilde{\alpha}) - m^{\dag}_{i})^{2},
\end{align*}
we just need to prove the following estimation
\begin{align}\label{convFor2proof}
\sup_{m_{s}^{\dag}\in S^{\tilde{\beta}}(R)\cap \Theta_{ss}^{\tilde{\beta}}(R,\delta)}\mathbb{E}_{0}\left\{
\sup_{\tilde{\alpha}\in \tilde{I}_{n}} \sum_{i=1}^{\infty}(\hat{m}_{i}(\tilde{\alpha}) - m^{\dag}_{i})^{2}
\right\} = O(\epsilon_{n}^{2}).
\end{align}
Inspired by Subsection 6.1 in \cite{Knapik2016}, the expectation can be split into square bias and variance terms as follows
\begin{align}\label{covTh1}
\sum_{i=1}^{\infty}\frac{i^{2+4\tilde{\alpha}}(m_{i}^{\dag})^{2}}{(i^{1+2\tilde{\alpha}}+n)^{2}} +
n\sum_{i=1}^{\infty}\frac{1}{(i^{1+2\tilde{\alpha}} + n)^{2}} = \text{I} + \text{II}.
\end{align}
Employing Lemma 8 in \cite{Knapik2016} (with $m = 0$, $\ell = 1$, $r = 1+2\tilde{\alpha}$ and $s = 0$), for large $n$, we find that
\begin{align}\label{covTh2}
\begin{split}
\text{II} \leq & \sum_{i=1}^{\infty}\frac{1}{i^{1+2\tilde{\alpha}} + n} \lesssim n^{-\frac{2\tilde{\alpha}}{1+2\tilde{\alpha}}} \\
\lesssim & \, n^{-\frac{2\tilde{\beta}-2C_2\log\log n/\log n}{1+2\tilde{\beta}-2C_2\log\log n/\log n}}
\lesssim (\log n)^{2C_{2}}n^{-\frac{2\tilde{\beta}}{1+2\tilde{\beta}}},
\end{split}
\end{align}
where we employed lower bound estimate in (\ref{estAlp1}). We should point out that $s$ is assumed to be positive in Lemma 8 in \cite{Knapik2016}. 
Through a straightforward verification, the estimates given in Lemma 8 in \cite{Knapik2016} still holds true when $s=0$.

For term $\text{I}$, it can be estimated by mimicking the procedures used in Subsection 6.1 in \cite{Knapik2016}
with lower bound estimate of $\tilde{\alpha}$ similar to (\ref{covTh2}). For reader's convenience, we provide some key steps in the following.
The summation in term $\text{I}$ will be divided into two parts that are $i>n^{1/(1+2\tilde{\beta})}$, 
$i\leq n^{1/(1+2\bar{\tilde{\alpha}}_n-2C_1\log\log n/\log n)}$. First we note that 
\begin{align}
\sum_{i>n^{1/(1+2\tilde{\beta})}} \frac{i^{2+4\tilde{\alpha}}(m_{i}^{\dag})^{2}}{(i^{1+2\tilde{\alpha}}+n)^{2}}
\leq \sum_{i>n^{1/(1+2\tilde{\beta})}} (m_{i}^{\dag})^{2} \leq n^{-\frac{2\tilde{\beta}}{(1+2\tilde{\beta})}}
\|m_{s}^{\dag}\|_{\tilde{\beta}}^{2}.
\end{align}
Next, elementary calculations indicate that for $\tilde{\alpha} > 0$ and $n\geq3$, the maximum of the function 
$i\mapsto i^{1+2\tilde{\alpha}}/\log i$ over the interval $[2, n^{1/(1+2\tilde{\alpha})}]$ is attained at 
$i = n^{1/(1+2\tilde{\alpha})}$, for $\tilde{\alpha} \leq \bar{\tilde{\alpha}}_n$. It follows that for $\tilde{\alpha} > 0$, 
\begin{align}
\begin{split}
& \sum_{i\leq n^{1/(1+2\tilde{\alpha})}}\frac{i^{2+4\tilde{\alpha}}(m_{i}^{\dag})^{2}}{(i^{1+2\tilde{\alpha}}+n)^2} \\
& = \frac{(m_{1}^{\dag})^{2}}{(1+n)^2} + \frac{1}{n^2}\sum_{2\leq i\leq n^{1/(1+2\tilde{\alpha})}}
\frac{((i^{1+2\tilde{\alpha}})/\log i)n^2 i^{1+2\tilde{\alpha}}(m_{i}^{\dag})^{2} \log i}{(i^{1+2\tilde{\alpha}} + n)^2} \\
& \leq \frac{(m_{1}^{\dag})^{2}}{(1+n)^2} + n^{-\frac{2\tilde{\alpha}}{1+2\tilde{\alpha}}}h_n(\tilde{\alpha}). 
\end{split}
\end{align}
Since $n^{1/(1+2\bar{\tilde{\alpha}}-2C_1\log\log n/\log n)} \leq n^{1/(1+2\tilde{\alpha})}$ 
for $\tilde{\alpha} \leq \bar{\tilde{\alpha}}_n -2C_1\log\log n/\log n$, the preceding implies that 
\begin{align}
\sup_{\tilde{\alpha}\in\tilde{I}_n}\sum_{i\leq n^{1/(1+2\bar{\tilde{\alpha}}_n-2C_1\log\log n/\log n)}}
\frac{i^{2+4\tilde{\alpha}}(m_{i}^{\dag})^{2}}{(i^{1+2\tilde{\alpha}}+n)^2} \lesssim \frac{1}{n^2} 
+ L n^{-\frac{2\underline{\tilde{\alpha}}_n}{1+2\underline{\tilde{\alpha}}_n}}\log^2 n,
\end{align}
where $L$ is the parameter given in (\ref{allhn3}). Relying on the lower bounded proved in Lemma \ref{lemmaUpDownEst1}, the right-hand
side of the above inequality is bounded by a constant times $n^{-2\tilde{\beta}/(1+2\tilde{\beta})}(\log n)^{2C_2 + 2}$. 
Because $\bar{\tilde{\alpha}}_n-C_1\log\log n/\log n \leq \tilde{\beta}$, the above two parts cover all of the possibilities,
which is a little bit different to the case in Subsection 6.1 in \cite{Knapik2016}. 
Combining estimates of terms I and II, we conclude the proof. 

\end{proof}

\textbf{Proof of Lemma \ref{weakSolLemma}}
\begin{proof}
We use the Lax-Milgram theorem \cite{Evans2010book} in the Hilbert space $H$. The coercive of $B$ can be illustrated as follows:
\begin{align*}
B(u,u) = n\|\mathcal{C}^{\alpha/2}\mathcal{T}^{*}u\|^{2} + \|u\|^{2} \geq \|u\|^{2}, \quad \forall\, u\in H.
\end{align*}
Using the statement (3) of Assumptions 2, we know that $\|\mathcal{C}^{\alpha/2}\mathcal{T}^{*}u\| \lesssim \|u\|_{\mathcal{H}^{-\alpha-(\Delta-1)}}$.
The continuity of the bilinear form $B$ can be derived as follows:
\begin{align*}
|B(u,v)| \leq n\|\mathcal{C}^{\alpha/2}\mathcal{T}^{*}u\|\|\mathcal{C}^{\alpha/2}\mathcal{T}^{*}v\| + \|u\|\|v\| \lesssim \|u\|\|v\|,
\quad \forall\, u,v \in H.
\end{align*}
So the first statement has been proved. For the second statement about regularities, we notice that
\begin{align*}
u = r - n\mathcal{M}(\alpha)u.
\end{align*}
Using statement (4) of Assumptions 2, we easily deduce $\mathcal{M}(\alpha)u \in \mathcal{H}^{2(\alpha+\Delta-1)}$ for $u\in H$.
Hence the right-hand side belongs to $\mathcal{H}^{t}$, which implies $u \in \mathcal{H}^{t}$ for $t \leq 2(\alpha+\Delta-1)$.
\end{proof}

\textbf{Proof of Lemma \ref{opEstLemma1}}
\begin{proof}
Let $h \in H$. Then $\mathcal{C}^{\frac{1}{2}(\alpha+\Delta-1)-\frac{s}{2}}h \in H$, since $\Delta \geq 1$ and $\alpha_{0} < s < \alpha$.
By Lemma \ref{weakSolLemma}, for $r = \mathcal{C}^{\frac{1}{2}(\alpha+\Delta-1)-\frac{s}{2}}h$, there exists a unique weak solution 
$u' \in H$ of (\ref{equation1}). 
Because $\mathcal{C}^{-\frac{1}{2}(\alpha+\Delta-1)+\frac{s}{2}}v \in H$ for $v \in \mathcal{H}^{(\alpha+\Delta-1)-s}$,
we conclude that for any $v \in \mathcal{H}^{(\alpha+\Delta-1)-s}$
\begin{align*}
& \Big(n^{\frac{1}{2}}\mathcal{C}^{\frac{\alpha}{2}}\mathcal{T}^{*}\mathcal{C}^{-\frac{1}{2}(\alpha+\Delta-1)+\frac{s}{2}}u,
n^{\frac{1}{2}}\mathcal{C}^{\frac{\alpha}{2}}\mathcal{T}^{*}\mathcal{C}^{-\frac{1}{2}(\alpha+\Delta-1)+\frac{s}{2}}v\Big) +  \\
& \quad\quad
\Big(\mathcal{C}^{-\frac{1}{2}(\alpha+\Delta-1)+\frac{s}{2}}u, \mathcal{C}^{-\frac{1}{2}(\alpha+\Delta-1)+\frac{s}{2}}v\Big) =
\Big(\mathcal{C}^{\frac{1}{2}(\alpha+\Delta-1)-\frac{s}{2}}h, \mathcal{C}^{-\frac{1}{2}(\alpha+\Delta-1)+\frac{s}{2}}v\Big),
\end{align*}
where $u = \mathcal{C}^{\frac{1}{2}(\alpha+\Delta-1)-\frac{s}{2}}u' \in \mathcal{H}^{(\alpha+\Delta-1)-s}$.
Choosing $v = u$, using statement (3) of Assumptions 2, for a generic constant $c>0$ and a small enough constant $\delta > 0$, we obtain
\begin{align*}
n\|u\|_{\mathcal{H}^{-s}}^{2} + & \|u\|_{\mathcal{H}^{(\alpha+\Delta-1)-s}}^{2} \leq c \|h\|\|u\| \\
\leq & c n^{-\frac{1}{2}\left(1 - \frac{s}{\alpha+\Delta-1}\right)}\|h\| \left( n^{\frac{1}{2}}\|u\|_{\mathcal{H}^{-s}}  \right)^{1-\frac{s}{\alpha+\Delta-1}}
\|u\|_{\mathcal{H}^{(\alpha+\Delta-1)-s}}^{\frac{s}{\alpha+\Delta-1}} \\
\leq & \frac{c}{2\delta}n^{-\left(1 - \frac{s}{\alpha+\Delta-1}\right)}\|h\|^{2}
+ \frac{c\delta}{2}\Big( n\|u\|_{\mathcal{H}^{-s}}^{2} + \|u\|_{\mathcal{H}^{(\alpha+\Delta-1)-s}}^{2} \Big),
\end{align*}
where interpolation inequality shown in Lemma \ref{L2Inter} has been employed. Then we obviously have
\begin{align*}
\|u\|_{\mathcal{H}^{(\alpha+\Delta-1)-s}} \lesssim n^{-\frac{1}{2}\left(1 - \frac{s}{\alpha+\Delta-1}\right)}\|h\|, \quad
\|u\|_{\mathcal{H}^{-s}} \lesssim n^{-1 + \frac{s}{2(\alpha+\Delta-1)}}\|h\|.
\end{align*}
Finally, inserting the above estimates into the following interpolation inequality
\begin{align*}
\|u\| \leq \|u\|_{\mathcal{H}^{-s}}^{1-\frac{s}{\alpha+\Delta-1}} \|u\|_{\mathcal{H}^{(\alpha+\Delta-1)-s}}^{\frac{s}{\alpha+\Delta-1}},
\end{align*}
we obtain
\begin{align*}
\|u\| \lesssim n^{-1 + \frac{s}{\alpha+\Delta-1}}\|h\|.
\end{align*}
Replacing $u = \mathcal{C}^{\frac{1}{2}(\alpha+\Delta-1)-\frac{s}{2}}(n\mathcal{M}(\alpha)+I)^{-1}\mathcal{C}^{\frac{1}{2}(\alpha+\Delta-1)-\frac{s}{2}}h$
gives the desired estimate (\ref{opEst1}).
\end{proof}

\textbf{Proof of Theorem \ref{converTheo2}}
\begin{proof}
	Denote $F_{dn}(\alpha) := \mathcal{C}_{0}^{\alpha + \Delta-1} + n^{-1}I$, then the conditional mean estimator for the artificial diagonal problem has the following form
	\begin{align}\label{cme1}
	\hat{m}_{dn}(\alpha) = m^{\dag} - \frac{1}{n}F_{dn}(\alpha)^{-1}m^{\dag} + \frac{1}{\sqrt{n}}\mathcal{C}_{0}^{\alpha + \Delta - 1}F_{dn}(\alpha)^{-1}\eta.
	\end{align}
	Because the parameters $\alpha$ and $\tilde{\alpha}$ are related with each other explicitly, 
	we use $\alpha$ instead of $\tilde{\alpha}$ here will not affect the conclusions obtained in Theorem \ref{converTheo1}.
	For the non-diagonal problem with $m \sim \mathcal{N}(0,\mathcal{M}(\alpha))$,
	the conditional mean estimator can be written as follows:
	\begin{align}\label{cme2}
	\hat{m}_{n}(\alpha) = m^{\dag} - \frac{1}{n}F_{n}(\alpha)^{-1}m^{\dag} + \frac{1}{\sqrt{n}}\mathcal{M}(\alpha)F_{n}(\alpha)^{-1}\eta,
	\end{align}
	where $F_{n}(\alpha) := \mathcal{M}(\alpha) + n^{-1}I$.
	Following (\ref{MarkovTrans1}) and recalling that $w_{n}$ and $\hat{m}_{n}(\alpha)$ are independent, we have
	\begin{align}
	\mathbb{E}_{0}\!\sup_{\alpha\in I_{n}}\mathbb{E}_{c, \alpha}\|w_n + \hat{m}_{n}(\alpha) - m^{\dag}\|^{2} \leq & \mathbb{E}_{0}\!\sup_{\alpha\in I_{n}}\mathbb{E}_{c,\alpha}\|w_{n}\|^{2} \!
	+ 2\mathbb{E}_{0}\!\sup_{\alpha\in I_{n}}\|\hat{m}_{n}(\alpha) - \hat{m}_{dn}(\alpha)\|^{2} \nonumber\\
	& + 2 \mathbb{E}_{0}\sup_{\alpha\in I_{n}}\|\hat{m}_{dn}(\alpha) - m^{\dag}\|^{2} \nonumber \\
	= & \text{I} + \text{II} + \text{III}. \nonumber
	\end{align}
	For term III, Theorem \ref{converTheo1} gives an appropriate estimate.
	For term I, let us denote the covariance operator as $\mathcal{C}_{n}^{p}$ which can be written as follows:
	\begin{align}\label{cme4}
	\mathcal{C}_{n}^{p} = (n\mathcal{M}(\alpha) + I)^{-1}\mathcal{M}(\alpha).
	\end{align}
	Considering $\Delta \geq 1$, both of $\mathcal{T}$ and $\mathcal{T}^{*}$ are bounded linear operators.
	Because $\mathcal{C}^{\alpha}$ is a trace class operator, we know that the operator $\mathcal{M}(\alpha)$ is a compact operator and
	can be diagonalized, which implies 
	$(n\mathcal{M}(\alpha) + I)^{-1}\mathcal{M}(\alpha) = \mathcal{M}(\alpha)^{1/2}(n\mathcal{M}(\alpha) + I)^{-1}\mathcal{M}(\alpha)^{1/2}$.
	Then for a white noise $\zeta$, we have the following estimate
	\begin{align}\label{cme5}
	\begin{split}
	\text{I} = & \sup_{\alpha\in I_{n}}\text{Tr}(\mathcal{C}_{n}^{p}) = \sup_{\tilde{\alpha}\in I_{n}}
	\text{Tr}(\mathcal{M}(\alpha)^{1/2}(n\mathcal{M}(\alpha) + I)^{-1}\mathcal{M}(\alpha)^{1/2}) \\
	= & \sup_{\alpha\in I_{n}}\mathbb{E}\Big\| \mathcal{M}(\alpha)^{1/2}(n\mathcal{M}(\alpha) + I)^{-1}\mathcal{M}(\alpha)^{1/2} \zeta \Big\|^{2} \\
	= & \sup_{\alpha\in I_{n}}\mathbb{E}\Big(\mathcal{C}^{s/2}\zeta,
	\mathcal{C}^{-s/2}\mathcal{M}(\alpha)^{1/2}(n\mathcal{M}(\alpha) + I)^{-1}\mathcal{M}(\alpha)^{1/2}\mathcal{C}^{-s/2}\mathcal{C}^{s/2}\zeta \Big)   \\
	\leq & \sup_{\alpha\in I_{n}}\Big\|\mathcal{C}^{-s/2}\mathcal{M}(\alpha)^{1/2}(n\mathcal{M}(\alpha)
	+ I)^{-1}\mathcal{M}(\alpha)^{1/2}\mathcal{C}^{-s/2}\Big\|_{\mathcal{B}(H)}
	\mathbb{E}\|\mathcal{C}^{s/2}\zeta\|^{2}.
	\end{split}
	\end{align}
	Choosing $s > \alpha_{0}$ in the above estimate, we know that $\mathbb{E}\|\mathcal{C}^{s/2}\zeta\|^{2} < \infty$ by Lemma 3.3 in \cite{Agapiou2013IP}.
	Through statement (5) of Assumptions 2, we have
	\begin{align}\label{jiacme1}
	\Big\| \mathcal{C}^{-\frac{1}{2}(\alpha+\Delta-1)+\frac{s}{2}}\mathcal{M}(\alpha)^{\frac{1}{2}}\mathcal{C}^{-\frac{s}{2}} \Big\|_{\mathcal{B}(H)} < \infty.
	\end{align}
	Combining estimates (\ref{cme5}), (\ref{jiacme1}) and Lemma \ref{opEstLemma1}, we obtain
	\begin{align}\label{jiacme2}
	\text{I} \lesssim \sup_{\alpha\in I_{n}}n^{-1 + \frac{s}{\alpha + \Delta - 1}}.
	\end{align}
	Taking $n$ large enough, we easily get
	\begin{align}\label{jiacme3}
	1 + \frac{\tilde{\epsilon} 2\tilde{\beta}}{1+2\tilde{\beta}+\tilde{\epsilon}}-\frac{1(1+\tilde{\epsilon}C_{2})}{1+2\tilde{\beta}+\tilde{\epsilon}}
	\frac{\log\log n}{\log n} > 1.
	\end{align}
	From the above estimate (\ref{jiacme3}) and the lower bound estimate (\ref{estAlp2}) of $\alpha$, we obtain
	\begin{align*}
	-1 + \frac{\alpha_{0}}{\alpha+\Delta-1} < -\frac{2\tilde{\beta}}{1+2\tilde{\beta}+\tilde{\epsilon}}.
	\end{align*}
	Hence, we can choose appropriate $s$ in (\ref{jiacme2}) to obtain
	\begin{align}\label{jiacme4}
	\text{I} \lesssim n^{-\frac{2\tilde{\beta}}{1+2\tilde{\beta}+\tilde{\epsilon}}}.
	\end{align}
	Now let us focus on term $\text{II}$. Obviously, we have
	\begin{align*}
	\hat{m}_{n}(\alpha) - \hat{m}_{dn}(\alpha) = & \frac{1}{n}\Big( F_{dn}(\alpha)^{-1} - F_{n}(\alpha)^{-1} \Big)m^{\dag} \\
	& \qquad\qquad
	+ \frac{1}{\sqrt{n}}\Big( \mathcal{M}(\alpha)F_{n}(\alpha)^{-1} - \mathcal{C}_{0}^{\alpha+\Delta-1}F_{dn}(\alpha)^{-1} \Big)\eta    \\
	= &\, \text{II}_{1} + \text{II}_{2}.
	\end{align*}
	Concerning term $\text{II}_{1}$, we find that
	\begin{align*}
	\Big\|\text{II}_{1}\Big\| = & \frac{1}{n}
	\Big\| (\mathcal{M}(\alpha) + n^{-1}I)^{-1}(\mathcal{C}_{0}^{\alpha+\Delta-1}-\mathcal{M}(\alpha))(\mathcal{C}_{0}^{\alpha+\Delta-1}+n^{-1}I)^{-1}m^{\dag} \Big\|   \\
	\leq & \frac{1}{n}\Big\| (\mathcal{M}(\alpha) + n^{-1}I)^{-1}\mathcal{C}_{0}^{\alpha+\Delta-1}(\mathcal{C}_{0}^{\alpha+\Delta-1}+n^{-1}I)^{-1}m^{\dag} \Big\|  \\
	& \qquad\quad\,\,\,\,
	+ \frac{1}{n} \Big\| (\mathcal{M}(\alpha) + n^{-1}I)^{-1}\mathcal{M}(\alpha)(\mathcal{C}_{0}^{\alpha+\Delta-1}+n^{-1}I)^{-1}m^{\dag} \Big\| \\
	= & \, \text{II}_{11} + \text{II}_{12}.
	\end{align*}
	Term $\text{II}_{11}$ can be estimated as follows:
	\begin{align*}
	\text{II}_{11} \lesssim &\, n^{-\frac{\tilde{\beta}}{1+2\tilde{\beta}+\tilde{\epsilon}}}
	\Big\| \mathcal{M}(\alpha)^{-\frac{\tilde{\beta}}{1+2\tilde{\beta}+\tilde{\epsilon}}}\mathcal{C}_{0}^{\alpha+\Delta-1}(\mathcal{C}_{0}^{\alpha+\Delta-1}+n^{-1}I)^{-1}m^{\dag} \Big\| \\
	\lesssim &\, n^{-\frac{\tilde{\beta}}{1+2\tilde{\beta}+\tilde{\epsilon}}}
	\Big\| \mathcal{M}(\alpha)^{-\frac{\tilde{\beta}}{1+2\tilde{\beta}+\tilde{\epsilon}}}
	\mathcal{C}_{0}^{\frac{\tilde{\beta}}{1+2\tilde{\beta}+\tilde{\epsilon}}(\alpha+\Delta-1)}\Big\|_{\mathcal{B}(H)}
	\Big\| \mathcal{C}_{0}^{-(\alpha+\Delta-1)\frac{\tilde{\beta}}{1+2\tilde{\beta}+\tilde{\epsilon}}}m^{\dag} \Big\|   \\
	\lesssim &\, n^{-\frac{\tilde{\beta}}{1+2\tilde{\beta}+\tilde{\epsilon}}}
	\Big\| \mathcal{M}(\alpha)^{-\frac{\tilde{\beta}}{1+2\tilde{\beta}+\tilde{\epsilon}}}
	\mathcal{C}^{\frac{\tilde{\beta}}{1+2\tilde{\beta}+\tilde{\epsilon}}(\alpha+\Delta-1)}\Big\|_{\mathcal{B}(H)}
	\Big\| \mathcal{C}^{-(\alpha+\Delta-1)\frac{\tilde{\beta}}{1+2\tilde{\beta}+\tilde{\epsilon}}}m^{\dag} \Big\|   \\
	\lesssim & \, n^{-\frac{\tilde{\beta}}{1+2\tilde{\beta}+\tilde{\epsilon}}}
	\Big\|m^{\dag}\Big\|_{\mathcal{H}^{(\alpha+\Delta-1)\frac{2\tilde{\beta}}{1+2\tilde{\beta}+\tilde{\epsilon}}}},
	\end{align*}
	where the third inequality is derived by using Lemma \ref{C0C} and the last inequality is
	obtained by employing statement (6) of Assumptions 2.
	Since $\alpha < \gamma+\frac{h}{2p}$, we have $(\alpha+\Delta-1)\frac{2\tilde{\beta}}{1+2\tilde{\beta}+\tilde{\epsilon}} \leq \gamma+\Delta-1$, which
	implies
	\begin{align}\label{jiacme6}
	\Big\|m^{\dag}\Big\|_{\mathcal{H}^{(\alpha+\Delta-1)\frac{2\tilde{\beta}}{1+2\tilde{\beta}+\tilde{\epsilon}}}} < \infty.
	\end{align}
	Hence, we obtain
	\begin{align}\label{jiacme5}
	\text{II}_{11} \lesssim n^{-\frac{\tilde{\beta}}{1+2\tilde{\beta}+\tilde{\epsilon}}}.
	\end{align}
	Similarly, term $\text{II}_{12}$ can be estimated as follows:
	\begin{align*}
	\text{II}_{12} \lesssim &\, n^{-\frac{\tilde{\beta}}{1+2\tilde{\beta}+\tilde{\epsilon}}}
	\Big\| \mathcal{M}(\alpha)^{1-\frac{\tilde{\beta}}{1+2\tilde{\beta}+\tilde{\epsilon}}}(\mathcal{C}_{0}^{\alpha+\Delta-1}+n^{-1}I)^{-1}m^{\dag} \Big\|   \\
	\lesssim &\, n^{-\frac{\tilde{\beta}}{1+2\tilde{\beta}+\tilde{\epsilon}}}
	\Big\| \mathcal{M}(\alpha)^{1-\frac{\tilde{\beta}}{1+2\tilde{\beta}+\tilde{\epsilon}}}
	\mathcal{C}^{(\alpha+\Delta-1)(1-\frac{\tilde{\beta}}{1+2\tilde{\beta}+\tilde{\epsilon}})} \Big\|_{\mathcal{B}(H)}
	\|m^{\dag}\|_{\mathcal{H}^{(\alpha+\Delta-1)\frac{2\tilde{\beta}}{1+2\tilde{\beta}+\tilde{\epsilon}}}}.
	\end{align*}
	Relying on statement (6) of Assumptions 2 and (\ref{jiacme6}), we obtain
	\begin{align}\label{jiacme7}
	\text{II}_{12} \lesssim n^{-\frac{\tilde{\beta}}{1+2\tilde{\beta}+\tilde{\epsilon}}}.
	\end{align}
	Combining estimates (\ref{jiacme5}) and (\ref{jiacme7}), we obtain
	\begin{align}\label{jiacme11}
	\|\text{II}_{1}\|^{2} \lesssim n^{-\frac{2\tilde{\beta}}{1+2\tilde{\beta}+\tilde{\epsilon}}}.
	\end{align}
	Through some simple calculations, term $\text{II}_{2}$ can be reformulated as follows:
	\begin{align*}
	\text{II}_{2} = \frac{1}{n^{3/2}}(\mathcal{M}(\alpha)+n^{-1}I)^{-1}(\mathcal{M}(\alpha)-\mathcal{C}_{0}^{\alpha+\Delta-1})
	(\mathcal{C}_{0}^{\alpha+\Delta-1}+n^{-1}I)^{-1}\eta.
	\end{align*}
	Then, we have
	\begin{align*}
	\mathbb{E}_{0}\big\| \text{II}_{2} \big\|^{2} \leq & \,
	n^{-3}\mathbb{E}_{0}\Big\|(\mathcal{M}(\alpha)+n^{-1}I)^{-1}\mathcal{C}_{0}^{\alpha+\Delta-1}(\mathcal{C}_{0}^{\alpha+\Delta-1}+n^{-1}I)^{-1}\eta\Big\|^{2} \\
	& \quad\quad
	+ n^{-3}\mathbb{E}_{0}\Big\|(\mathcal{M}(\alpha)+n^{-1}I)^{-1}\mathcal{M}(\alpha)(\mathcal{C}_{0}^{\alpha+\Delta-1}+n^{-1}I)^{-1}\eta\Big\|^{2} \\
	\leq &\, n^{-1}\mathbb{E}_{0}\Big\|\mathcal{C}_{0}^{\alpha+\Delta-1}(\mathcal{C}_{0}^{\alpha+\Delta-1}+n^{-1}I)^{-1}\eta\Big\|^{2} \\
	& \quad\quad
	+ n^{-1}\mathbb{E}_{0}\Big\| \mathcal{M}(\alpha)(\mathcal{C}_{0}^{\alpha+\Delta-1}+n^{-1}I)^{-1}\eta \Big\|^{2} \\
	= & \, \text{II}_{21} + \text{II}_{22}.
	\end{align*}
	For term $\text{II}_{21}$, we have the following estimates
	\begin{align}\label{jiacme8}
	\begin{split}
	\text{II}_{21} \lesssim & \, n^{-\frac{2\tilde{\beta}}{1+2\tilde{\beta}+\tilde{\epsilon}}}
	\mathbb{E}_{0}\Big\| \mathcal{C}_{0}^{(\alpha+\Delta-1)(\frac{1}{2}+\frac{\tilde{\beta}}{1+2\tilde{\beta}+\tilde{\epsilon}})}\eta \Big\|^{2} \\
	\lesssim & \, n^{-\frac{2\tilde{\beta}}{1+2\tilde{\beta}+\tilde{\epsilon}}}
	\mathbb{E}_{0}\Big\| \mathcal{C}^{(\alpha+\Delta-1)(\frac{1}{2}+\frac{\tilde{\beta}}{1+2\tilde{\beta}+\tilde{\epsilon}})}\eta \Big\|^{2}.
	\end{split}
	\end{align}
	Because $(\alpha+\Delta-1)\left(1+\frac{2\tilde{\beta}}{1+2\tilde{\beta}+\tilde{\epsilon}}\right) > \alpha_{0}$ for $n$ large enough, from (\ref{jiacme8}) we obtain
	\begin{align}\label{jiacme9}
	\text{II}_{21} \lesssim \, n^{-\frac{2\tilde{\beta}}{1+2\tilde{\beta}+\tilde{\epsilon}}},
	\end{align}
	where we used Lemma 3.3 in \cite{Agapiou2013IP}.
	For term $\text{II}_{22}$, we have
	\begin{align}\label{jiacme10}
	\begin{split}
	\text{II}_{22} \leq &\, n^{-1}\Big\| \mathcal{M}(\alpha)\mathcal{C}_{0}^{-(\alpha+\Delta-1)}\Big\|^{2} \mathbb{E}_{0}
	\Big\|\mathcal{C}_{0}^{(\alpha+\Delta-1)}(\mathcal{C}_{0}^{\alpha+\Delta-1}+n^{-1}I)^{-1}\eta \Big\|^{2}    \\
	\leq &\, n^{-1}\Big\| \mathcal{M}(\alpha)\mathcal{C}^{-(\alpha+\Delta-1)}\Big\|^{2} \mathbb{E}_{0}
	\Big\|\mathcal{C}_{0}^{(\alpha+\Delta-1)}(\mathcal{C}_{0}^{\alpha+\Delta-1}+n^{-1}I)^{-1}\eta \Big\|^{2}    \\
	\lesssim &\, \text{II}_{21} \lesssim n^{-\frac{2\tilde{\beta}}{1+2\tilde{\beta}+\tilde{\epsilon}}},
	\end{split}
	\end{align}
	where statement (7) of Assumptions 2 has been employed.
	Combining estimates (\ref{jiacme9}) and (\ref{jiacme10}), we obtain
	\begin{align}\label{jiacme12}
	\mathbb{E}_{0}\|\text{II}_{2}\|^{2} \lesssim n^{-\frac{2\tilde{\beta}}{1+2\tilde{\beta}+\tilde{\epsilon}}}.
	\end{align}
	Combining estimations of $\text{I}, \text{II}$ and $\text{III}$ and recalling (\ref{MarkovTrans1}), we finish the proof.
\end{proof}

\textbf{Proof of Theorem \ref{mainTheorem}}
\begin{proof}
	\textbf{Step 1}.
	Let us firstly assume that 
	\begin{align}\label{mainTh1}
	\mathbb{E}_{0}\mu_{\hat{\alpha}_{n}}^{d}(S_{n}^{c}) \rightarrow 0
	\end{align}
	holds true. Let $k_{n}$, $\rho_{n}$ and $c$ in the definition of $S_{n}$ be fixed. For any $u \in S_{n}$, we have
	\begin{align}\label{mainTh2}
	\begin{split}
	\|u\|^{2} = & \sum_{j = 1}^{\infty}u_{j}^{2} = \sum_{j \leq k_{n}}u_{j}^{2} + \sum_{j > k_{n}}u_{j}^{2} \\
	\leq & \sum_{j \leq k_{n}}u_{j}^{2} + c\rho_{n}^{2} = \sum_{j\leq k_{n}}\lambda_{j}^{2(\Delta-1)}\lambda_{j}^{-2(\Delta-1)}u_{j}^{2} + c\rho_{n}^{2}    \\
	\leq & \lambda_{k_{n}}^{-2(\Delta-1)}\sum_{j\leq k_{n}}\lambda_{j}^{2(\Delta-1)}u_{j}^{2} + c\rho_{n}^{2}   \\
	\leq & \lambda_{k_{n}}^{-2(\Delta-1)}\|u\|_{\mathcal{H}^{-(\Delta-1)}}^{2} + c\rho_{n}^{2}.
	\end{split}
	\end{align}
	Denote $u^{\dag}_{s}:=\{u^{\dag}_{j}\}_{j=1}^{\infty}$ with $u^{\dag}_{j}:=(u^{\dag},\phi_{j})$ for $j=1,2,\ldots$.
	Denote $u_{n}^{\dag}$ be function related to the projection of $u_{s}^{\dag}$ on the first $k_{n}$ coordinates,
	that is, $u_{n}^{\dag} = \sum_{j=1}^{k_{n}}u_{j}^{\dag}\phi_{j}$.
	Then, we have
	\begin{align}\label{mainTh3}
	\begin{split}
	\|u_{n}^{\dag} - u^{\dag}\|^{2} = & \sum_{j = 1}^{\infty}(u_{n,j}^{\dag} - u^{\dag}_{j})^{2} = \sum_{j > k_{n}}(u^{\dag}_{j})^{2} \\
	= & \sum_{j > k_{n}}\lambda_{j}^{-2\gamma}(u_{j}^{\dag})^{2}\lambda_{j}^{2\gamma} \leq \lambda_{k_{n}}^{2\gamma}\|u^{\dag}\|_{\mathcal{H}^{\gamma}}^{2},
	\end{split}
	\end{align}
	and
	\begin{align}\label{mainTh4}
	\begin{split}
	\|u_{n}^{\dag} - u^{\dag}\|_{\mathcal{H}^{-(\Delta-1)}}^{2} = & \sum_{j=1}^{\infty}\lambda_{j}^{2(\Delta-1)}(u_{n,j}^{\dag} - u_{j}^{\dag})^{2} \\
	= & \sum_{j > k_{n}}\lambda_{j}^{2(\Delta-1)}(u_{j}^{\dag})^{2}\lambda_{j}^{-2\gamma}\lambda_{j}^{2\gamma}  \\
	\leq & \lambda_{k_{n}}^{2(\Delta-1+\gamma)}\|u^{\dag}\|_{\mathcal{H}^{\gamma}}^{2}.
	\end{split}
	\end{align}
	Using estimates (\ref{mainTh3}) and (\ref{mainTh4}), we find that
	\begin{align}\label{mainTh5}
	\begin{split}
	\|u - & u^{\dag}\| \leq \|u - u_{n}^{\dag}\| + \|u_{n}^{\dag} - u^{\dag}\|  \\
	\leq & \lambda_{k_{n}}^{-(\Delta-1)}\|u - u_{n}^{\dag}\|_{\mathcal{H}^{-(\Delta-1)}} + \sqrt{c}\rho_{n} + \lambda_{k_{n}}^{\gamma}\|u^{\dag}\|_{\mathcal{H}^{\gamma}}  \\
	\leq & \lambda_{k_{n}}^{-(\Delta-1)}\Big[
	\|u - u^{\dag}\|_{\mathcal{H}^{-(\Delta-1)}} + \lambda_{k_{n}}^{\Delta-1+\gamma}\|u^{\dag}\|_{\mathcal{H}^{\gamma}} \Big] + \sqrt{c}\rho_{n}
	+ \lambda_{k_{n}}^{\gamma}\|u^{\dag}\|_{\mathcal{H}^{\gamma}}  \\
	\leq & \lambda_{k_{n}}^{-(\Delta-1)}\|u-u^{\dag}\|_{\mathcal{H}^{-(\Delta-1)}} + \sqrt{c}\rho_{n} + 2\lambda_{k_{n}}^{\gamma}\|u^{\dag}\|_{\mathcal{H}^{\gamma}}    \\
	\lesssim & \lambda_{k_{n}}^{-(\Delta-1)}\|u-u^{\dag}\|_{\mathcal{H}^{-(\Delta-1)}} + \rho_{n} + \lambda_{k_{n}}^{\gamma} \\
	\lesssim & \lambda_{k_{n}}^{-(\Delta-1)}\|\mathcal{T}u-\mathcal{T}u^{\dag}\| + \rho_{n} + \lambda_{k_{n}}^{\gamma}.
	\end{split}
	\end{align}
	For a large positive constant $M > 0$ (will be specified later), we denote $\tilde{\epsilon} = \epsilon/M$. 
	Denote $L_n = (\log n)^{C_2 + 1}$ as in Theorem \ref{converTheo2}. Let
	\begin{align}\label{mainTh6}
	k_{n} = n^{\frac{1}{1+\frac{2p}{h}(\gamma+\Delta-1)+\tilde{\epsilon}}}, \quad
	\rho_{n} = L_{n}n^{-\frac{\frac{1}{2}\left( \frac{2p}{h}\hat{\alpha}_{n} - 1 - \tilde{\epsilon} \right)}{1+\frac{2p}{h}(\gamma+\Delta-1)+\tilde{\epsilon}}},
	\end{align}
	and $\|\mathcal{T}u-\mathcal{T}u^{\dag}\| \leq M_{n}\tilde{\epsilon}_{n}$ with $\tilde{\epsilon}_{n}$ defined as in Theorem \ref{converTheo2}, we find that
	\begin{align}\label{mainTh7}
	\begin{split}
	\|u - u^{\dag}\| \lesssim & M_{n}L_{n}k_{n}^{\frac{p}{h}(\Delta-1)} n^{-\frac{\frac{p}{h}(\gamma+\Delta-1)}{1+\frac{2p}{h}(\gamma+\Delta-1)+\tilde{\epsilon}}} +
	L_{n}n^{-\frac{\frac{1}{2}\left( \frac{2p}{h}\hat{\alpha}_{n} - 1 - \tilde{\epsilon}\right)}{1+\frac{2p}{h}(\gamma+\Delta-1)+\tilde{\epsilon}}}
	+ k_{n}^{-\frac{p}{h}\gamma} \\
	\lesssim & M_{n}L_{n}n^{-\frac{\frac{1}{2}\left( \frac{2p}{h}\hat{\alpha}_{n} - 1 \right) - \frac{1}{2}\tilde{\epsilon}}{1+\frac{2p}{h}(\gamma+\Delta-1)+\tilde{\epsilon}}}.
	\end{split}
	\end{align}
	Applying estimate (\ref{estAlp2}) to the last line of (\ref{mainTh7}), we finally obtain
	\begin{align}\label{mainTh8}
	\|u - u^{\dag}\| \lesssim M_{n}L_{n}(\log n)^{C_{3}}n^{-\frac{p\gamma - \frac{h}{2}\tilde{\epsilon}}{h+2p(\gamma+\Delta-1) + h\tilde{\epsilon}}}.
	\end{align}
	Taking $M > 0$ large enough such that
	\begin{align}\label{mainTh92}
	\frac{h^{2} + 2ph(\gamma+\Delta-1)+h\epsilon}{2p\gamma} + h \leq M,
	\end{align}
	then we have
	\begin{align*}
	n^{-\frac{p\gamma-\frac{h}{2}\tilde{\epsilon}}{h+2p(\gamma+\Delta-1)+h\tilde{\epsilon}}}\leq
	n^{-\frac{p\gamma}{h+2p(\gamma+\Delta-1)+\epsilon}}.
	\end{align*}
	Recalling Theorem \ref{converTheo2} and Theorem 2.1 in \cite{Salomond2018B}, the proof is completed.
	
	\textbf{Step 2}. In this step, we aim to prove (\ref{mainTh1}).
	In the following, we use the values of $k_{n}$ and $\rho_{n}$ given in (\ref{mainTh6}).
    Let $W_1, W_2, \cdots$ be independent standard normal random variables, then we have 
	\begin{align*}
	\mu_{n}^{0}(\mathcal{S}_{n}^{c}) = \text{Pr}\left( \sum_{i>k_{n}}\lambda_{i}^{2\hat{\alpha}_{n}}W_{i}^{2} > c\rho_{n}^2 \right).
	\end{align*}
	For some $t>0$
	\begin{align*}
	\text{Pr}\left( \sum_{i>k_{n}}\lambda_{i}^{2\hat{\alpha}_{n}}W_{i}^{2} > c\rho_{n}^2 \right) = &
	\text{Pr}\left( \exp\left(t\sum_{i>k_{n}}\lambda_{i}^{2\hat{\alpha}_{n}}W_{i}^{2}\right) > \exp(tc\rho_{n}^2) \right)  \\
	\leq & \exp\left( -tc\rho_n^2 \right)\mathbb{E}\exp\left( t\sum_{i>k_n}\lambda_i^{2\hat{\alpha}_n}W_i^2 \right)	\\
	= & \exp\left( -tc\rho_n^2 \right) \prod_{i>k_n}\left( 1-2t\lambda_{i}^{2\hat{\alpha}_n} \right)^{-1/2}.
	\end{align*}
	Taking the logarithm of the right-hand side of the above inequality, we have 
	\begin{align*}
	-tc\rho_n^2 - \frac{1}{2}\sum_{i > k_n}\log\big( 1-2t\lambda_i^{2\hat{\alpha}_n} \big) \leq 
	-tc\rho_n^2 + \frac{1}{2}\sum_{i>k_n}\frac{2t\lambda_i^{2\hat{\alpha}_n}}{1-2t\lambda_{i}^{2\hat{\alpha}_n}},
	\end{align*}
	where we used the elementary inequality $\log(1-y)\geq -y/(1-y)$. 
	For the second term on the right-hand side, we have 
	\begin{align*}
	\frac{1}{2}\sum_{i>k_n}\frac{2t\lambda_i^{2\hat{\alpha}_n}}{1-2t\lambda_{i}^{2\hat{\alpha}_n}} \leq 
	\frac{tC^{2\hat{\alpha}_n}}{1-2tC^{2\hat{\alpha}_n}k_n^{-2p\hat{\alpha}_n/h}}\sum_{i>k_n}i^{-\frac{2p}{h}\hat{\alpha}_n}, 
	\end{align*}
	where $C$ is the constant appeared in Assumptions 1 such that $C^{-1}\leq c_i \leq C$ with $i=1,2,\cdots$. 
	Since $x^{-2p\hat{\alpha}_n/h}$ is decreasing, we find that 
	\begin{align*}
	\sum_{i > k_n} i^{-2p\hat{\alpha}_n/h}\leq \int_{k_n}^{\infty} x^{-2p\hat{\alpha}_n/h} dx + k_n^{-2p\hat{\alpha}_n/h}
	\leq k_n^{1-\frac{2p\hat{\alpha}_n}{h}}\frac{2p\hat{\alpha}_n/h}{2p\hat{\alpha}_n/h-1}. 
	\end{align*}
	Considering the above estimates and taking $t=\frac{k_n^{2p\hat{\alpha}_n/h}}{4C^{2\hat{\alpha}_n}}$, we obtain
	\begin{align*}
	\mu_n^0(\mathcal{S}_n^c) \leq \exp\left( 
	-\frac{c}{4C^{2\hat{\alpha}_n}}\rho_n^2 k_{n}^{2p\hat{\alpha}_n/h} + 
	\frac{1}{2}k_n\frac{2p\hat{\alpha}_n/h}{2p\hat{\alpha}_n/h-1}
	\right).
	\end{align*}
	Considering the explicit form of $\rho_n$, $k_n$ and taking $n$ large enough, we have
	\begin{align*}
	\frac{c}{4}\rho_n^2 k_{n}^{2p\hat{\alpha}_n/h} \geq C^{2\hat{\alpha}_n}\frac{2p\hat{\alpha}_n/h}{2p\hat{\alpha}_n/h-1}.
	\end{align*}
	Finally, we obtain
	\begin{align}\label{mainTh9}
	\begin{split}
	\mu_{n}^{0}(S_{n}^{c}) \leq & \exp\left( -\frac{c}{8C^{2\hat{\alpha}_n}}\rho_{n}^{2}k_{n}^{\frac{2p}{h}\hat{\alpha}_{n}} \right) \\
	= & \exp\left( -\frac{c}{8C^{2\hat{\alpha}_n}}L_{n}^{2}n^{\frac{1+\tilde{\epsilon}}{1+\frac{2p}{h}(\gamma+\Delta-1)+\tilde{\epsilon}}} \right).
	\end{split}
	\end{align}
	Choose $\tilde{\epsilon}_{n}$ as in Theorem \ref{converTheo2} and notice that for a constant $C > 0$,
	\begin{align}\label{term51}
	\begin{split}
	& \mu_{n}^{0}(B_{n}(\mathcal{A}^{-1}u^{\dag},\tilde{\epsilon}_{n})) \geq \mu_{n}^{0}(u\,:\, \|\mathcal{T}u-\mathcal{T}u^{\dag}\|^{2} \leq \tilde{\epsilon}_{n}^{2}) \\
	& \qquad \geq \mu_{n}^{0}\Big(u\,:\, \sum_{i=1}^{\infty}i^{-\frac{2p}{h}(\Delta-1)}c_i^{-2\hat{\alpha}_n}(u_i-u_i^{\dag})^2 \leq C^{-1}\tilde{\epsilon}_{n}^{2}\Big) \geq \mu_{n}^{0}(A_1)
	\mu_{n}^{0}(A_2).
	\end{split}
	\end{align}
	where
	\begin{align}
	A_1 & = \left\{ u:\sum_{i=1}^{N}i^{-\frac{2p}{h}(\Delta-1)}c_i^{-2\hat{\alpha}_n}(u_{i} - u_{i}^{\dag})^{2} \leq \frac{\tilde{\epsilon}_{n}^{2}}{2C} \right\}, \\
	A_2 & = \left\{ u:\sum_{i=N+1}^{\infty}i^{-\frac{2p}{h}(\Delta-1)}c_i^{-2\hat{\alpha}_n}(u_{i} - u_{i}^{\dag})^{2} \leq \frac{\tilde{\epsilon}_{n}^{2}}{2C} \right\}.
	\end{align}
	In the following, we will use (\ref{estAlp2}) frequently during the estimates related to $\hat{\alpha}_n$ 
	without mentioning it each time. We have
	\begin{align}\label{ddd1}
	\begin{split}
	& \sum_{i=N+1}^{\infty}i^{-\frac{2p}{h}(\Delta-1)}c_i^{-2\hat{\alpha}_n}(u_{i} - u_{i}^{\dag})^{2} \\
	& \qquad \leq 
	2\sum_{i=N+1}^{\infty}i^{-\frac{2p}{h}(\Delta-1)}c_i^{-2\hat{\alpha}_n}u_{i}^2 + 2\sum_{i=N+1}^{\infty}i^{-\frac{2p}{h}(\Delta-1)}c_i^{-2\hat{\alpha}_n}(u_{i}^{\dag})^2. 
	\end{split}
	\end{align}
	The second sum in the display above is less than or equal 
	\begin{align}
	2\tilde{C}N^{-\frac{2p}{h}(\gamma+\Delta-1)} \|u^{\dag}\|_{\mathcal{H}^{\gamma}}^2 < \frac{\tilde{\epsilon}_{n}^{2}}{4C}, 
	\end{align}
	whenever $N > N_1=(8\tilde{C}\|u^{\dag}\|_{\mathcal{H}^{\gamma}}^2)^{\frac{1}{\frac{2p}{h}(\gamma+\Delta-1)}} \tilde{\epsilon}_{n}^{-\frac{1}{\frac{p}{h}(\gamma+\Delta-1)}}$. 
	Here and in the following, $\tilde{C}$ will be some general constant that may be different from line to line. 
	By Chebyshev's inequality, the first sum on the right-hand side of (\ref{ddd1}) is less than $\frac{\tilde{\epsilon}_{n}^{2}}{4C}$ with 
	probability at least
	\begin{align}
	\begin{split}
	& 1-\frac{8C}{\tilde{\epsilon}_n^2}\sum_{i=N+1}^{\infty}i^{-\frac{2p}{h}(\hat{\alpha}_n + \Delta - 1)}  \\
	& \qquad\quad \geq 
	1 - \frac{8C}{\tilde{\epsilon}_n^2 \left( \frac{2p}{h}(\hat{\alpha}_n+\Delta-1)-1 \right) N^{\frac{2p}{h}(\hat{\alpha}_n+\Delta-1)-1}} > 1/2
	\end{split}
	\end{align}
	if $N > N_2 = \tilde{C}(p,h,\gamma,\ell,\beta)\tilde{\epsilon}_n^{-\frac{2}{\frac{2p}{h}(\hat{\alpha}_n+\Delta-1)-1}}$.
	To bound the first term in (\ref{term51}), we apply Lemma 6.2 in \cite{Belitser2003AS} with $\xi_i = i^{-\frac{p}{h}(\Delta-1)}c_i^{-\hat{\alpha}_n}u_i^{\dag}$
	and $\delta^2 = \tilde{\epsilon}_n^2/2C$. Note that 
	\begin{align}
	\begin{split}
	\sum_{i=1}^{N} i^{\frac{2p}{h}(\hat{\alpha}_n +\Delta-1)}\xi_i^2 \leq 
	\tilde{C}N^{\frac{2p}{h}(\hat{\alpha}_n-\gamma)}\|u^{\dag}\|_{\mathcal{H}^{\gamma}}^2.
	\end{split}
	\end{align}
	Following the arguments given in the proof of Lemma 5.1 in \cite{Salomond2018B}, we obtain
	\begin{align}
	\begin{split}
	& \mu_n^0(u : \|\mathcal{T}u - \mathcal{T}u^{\dag}\| \leq \tilde{\epsilon}_n)  \\
	& \quad \geq 
	\frac{1}{8}\exp\left(\!-\left( \frac{2p}{h}(\hat{\alpha}_n + \Delta - 1) + \frac{\log 2}{2} \right)N \right)
	\exp\left(\! -\tilde{C}N^{\frac{2p}{h}(\hat{\alpha}_n-\gamma)}\|u^{\dag}\|_{\mathcal{H}^{\gamma}}^2 \right)
	\end{split}
	\end{align}
	when $N > N_3 = \tilde{C}(p,h,\gamma,\ell,\beta)\tilde{\epsilon}_n^{-\frac{2}{\frac{2p}{h}(\hat{\alpha}_n+\Delta-1)-1}}$.
	Since $\frac{2p}{h}(\hat{\alpha}_n-\gamma) \leq 1$, there is a constant $C_4$ such that 
	\begin{align}\label{mainTh10}
	\begin{split}
	& \mu_{n}^{0}(B_{n}(\mathcal{A}^{-1}u^{\dag},\tilde{\epsilon}_{n})) \geq \tilde{C}
	\exp\left( -C_{4}N \right).
	\end{split}
	\end{align}
	Considering $N > \max\{N_1, N_2, N_3\}$, we have 
	\begin{align}\label{mainTh101}
	\mu_{n}^{0}(B_{n}(\mathcal{A}^{-1}u^{\dag},\tilde{\epsilon}_{n})) \geq \tilde{C}
	\exp\left( - C_5 \tilde{\epsilon}_n^{-\frac{2}{\frac{2p}{h}(\hat{\alpha}_n+\Delta-1)-1}} \right),
	\end{align}
	where $C_5$ is a positive constant. 
	Combining (\ref{mainTh9}) and (\ref{mainTh101}), we arrive at
	\begin{align}\label{mainTh11}
	\begin{split}
	\frac{\mu_{n}^{0}(S_{n}^{c})}{\mu_{n}^{0}(B_{n}(\mathcal{A}^{-1}u^{\dag},\tilde{\epsilon}_{n}))} \lesssim
	\exp\left( -\Big( \frac{c}{8C^{\gamma + h/2p}} - C_{5}L_{n}^{-\frac{h}{p(\hat{\alpha}_n+\Delta-1)}-2}n^{r} \Big)
	n\tilde{\epsilon}_{n}^{2} \right),
	\end{split}
	\end{align}
	where $r = -\frac{1+\tilde{\epsilon}}{1+2\tilde{\beta}+\tilde{\epsilon}} + \frac{2\tilde{\beta}}{(1+2\tilde{\beta}+\tilde{\epsilon})\left(\frac{2p}{h}(\hat{\alpha}_n+\Delta-1)-1\right)}$. 
	Through direct calculations, we can verify that $r \leq 0$ when $n$ is taken to be a large number such that 
	$\log n > \frac{hC_2}{p\tilde{\epsilon}(\gamma+\Delta-1)}\log\log n$. 
	Choosing $c$ in the definition of $S_{n}$ large enough, we obtain (\ref{mainTh1}) by using Lemma \ref{cond1lemma}, which 
	completes the proof. 
\end{proof}

\textbf{Proof of (\ref{set3})}
\begin{proof}
Firstly, let us denote $W_{z}$ as the white noise mapping for $z\in H$ \cite{Prato2006AIIDA}.
Recalling the Cameron-Martin formula, we obtain
\begin{align}\label{equil1}
\frac{d\mu_{\ell}'}{d\mu_{\ell}}(d) = \exp\left( -\frac{n}{2}\big(\|\mathcal{T}u \|^{2} - \|\mathcal{T}u^{\dag}\|^{2}\big) 
+ W_{\sqrt{n}\mathcal{T}(u-u^{\dag})}(d) \right).
\end{align}
Taking logarithm with respect to the above equality (\ref{equil1}) and integrating the obtained equality, we obtain
\begin{align}\label{equil2}
\begin{split}
-\int\log\frac{d\mu_{\ell}'}{d\mu_{\ell}}d\mu_{\ell} = &
\int \frac{n}{2}\|\mathcal{T}u\|^{2} - \frac{n}{2}\|\mathcal{T}u^{\dag}\|^{2} - W_{\sqrt{n}\mathcal{T}(u-u^{\dag})}(\cdot)d\mu_{\ell} \\
= & \frac{n}{2}\|\mathcal{T}(u - u^{\dag})\|^{2},
\end{split}
\end{align}
where the properties of the white noise mapping are employed.
To complete the proof, we should notice that
\begin{align}\label{equil3}
\begin{split}
\int \Big| \log\frac{d\mu_{\ell}'}{d\mu_{\ell}} - & \int\log\frac{d\mu_{\ell}'}{d\mu_{\ell}}d\mu_{\ell} \Big|^{2}d\mu_{\ell} \\
& = \int\left( \log\frac{d\mu_{\ell}'}{d\mu_{\ell}} \right)^{2}d\mu_{\ell} - \left(\int\log\frac{d\mu_{\ell}'}{d\mu_{\ell}}d\mu_{\ell}\right)^{2}.
\end{split}
\end{align}
For the first term on the right hand side of (\ref{equil3}), we have
\begin{align}\label{equil4}
\begin{split}
\int\left( \log\frac{d\mu_{\ell}'}{d\mu_{\ell}} \right)^{2}d\mu_{\ell} = &
\frac{n^{2}}{4}\big(\|\mathcal{T}u \|^{2} - \|\mathcal{T}u^{\dag}\|^{2}\big)^{2} 
+ \int \left( W_{\sqrt{n}\mathcal{T}(u-u^{\dag})}(\cdot) \right)^{2}d\mu_{\ell} \\
& - \int n\big(\|\mathcal{T}u \|^{2} - \|\mathcal{T}u^{\dag}\|^{2}\big)W_{\sqrt{n}\mathcal{T}(u-u^{\dag})}(\cdot)d\mu_{\ell} \\
= & \frac{n^{2}}{4}\|\mathcal{T}(u-u^{\dag})\|^{4} + n\|\mathcal{T}(u-u^{\dag})\|^{2}.
\end{split}
\end{align}
Because
\begin{align*}
\left(\int\log\frac{d\mu_{\ell}'}{d\mu_{\ell}}d\mu_{\ell}\right)^{2} = \frac{n^{2}}{4}\|\mathcal{T}(u-u^{\dag})\|^{4},
\end{align*}
we find that
\begin{align}\label{equil5}
\int \Big| \log\frac{d\mu_{\ell}'}{d\mu_{\ell}} - & \int\log\frac{d\mu_{\ell}'}{d\mu_{\ell}}d\mu_{\ell} \Big|^{2}d\mu_{\ell} =
n\|\mathcal{T}(u-u^{\dag})\|^{2}.
\end{align}
Combining equalities (\ref{equil2}) and (\ref{equil5}), we obtain
\begin{align*}
B_{n}(\mathcal{A}^{-1}u^{\dag},\epsilon) = \Big\{ u\in H \,:\, \|\mathcal{T}(u - u^{\dag})\|^{2} \leq \epsilon^{2} \Big\}.
\end{align*}
\end{proof}



\section*{Acknowledgments}
The authors would like to thank the anonymous
referees for their comments and suggestions, which helped to improve the paper significantly. 
This work was partially supported by the NSFC under the grants Nos. 11871392, 11771347, and
the National Key R\&D Program of China grant No. 2018YFC0603501,
and the National Science and Technology Major Project under grant Nos. 2016ZX05024-001-007 and 2017ZX05069.

\bibliographystyle{plain}
\bibliography{references}

\end{document}